\documentclass{amsart}
\usepackage{amsmath, amssymb,graphicx,amscd}
\usepackage{graphicx}
\usepackage{xypic}
\usepackage{hyperref, setspace}

\usepackage{colortbl}

\usepackage[all, knot]{xy}
\xyoption{arc}

\newtheorem*{thm-class}{Theorem \ref{thm-statecycle-classification}}
\newtheorem*{thm-lifting-qp}{Theorem \ref{lifting-q-p}}
\newtheorem*{thm-lifting-lee}{Theorem \ref{lifting-lee}}
\newtheorem{thm}{Theorem}[section]
\newtheorem{lemma}[thm]{Lemma}
\newtheorem{prop}[thm]{Proposition}
\newtheorem{defn}[thm]{Definition}
\newtheorem{cor}[thm]{Corollary}
\newtheorem{conj}{Conjecture}

\theoremstyle{remark}
\newtheorem{example}{Example}[section]
\newtheorem{remark}{Remark}[thm]
\newtheorem*{induction}{Induction Hypothesis}
\newtheorem{claim}{Claim}[thm]

\newcommand{\Q}{{\mathbb{Q}}}
\newcommand{\sz}{{\mathfrak{s}_0}}
\newcommand{\Kh}{{Kh^{i,j}}}
\DeclareMathOperator{\Span}{Span}

\DeclareMathOperator{\im}{im}

\definecolor{orange}{rgb}{1,0.84,0}
\definecolor{beige}{rgb}{0.93,0.87,.51}
\newcommand{\cy}{\cellcolor{yellow}}
\newcommand{\co}{\cellcolor{orange}}
\newcommand{\cb}{\cellcolor{beige}}

\newcolumntype{q}{>{\bfseries}c}

\title[State cycles, quasipositive modification, and constructing H-thick knots]{State cycles, quasipositive modification, and constructing H-thick knots in Khovanov homology}
\author{Andrew Elliott}
\thanks{\noindent Partially supported by the National Science 
Foundation  DMS-0706929}
\address{Department of Mathematics\\Rice University\\6100 S. Main Street\\Houston, TX 77005 }
\email{elflord@rice.edu}

\begin{document}

\begin{abstract}
We study Khovanov homology classes which have state cycle representatives, and examine how they interact with Jacobsson homomorphisms and Lee's map $\Phi$.  As an application, we describe a general procedure, quasipositive modification, for constructing H-thick knots in rational Khovanov homology.  Moreover, we show that specific families of such knots cannot be detected by Khovanov's thickness criteria.  We also exhibit a sequence of prime links related by quasipositive modification whose width is increasing.
\end{abstract}

\maketitle

\section{Introduction}

Khovanov homology is a bigraded link homology whose graded Euler characteristic is a normalization of the Jones polynomial.  As a knot invariant, it is stronger than the Jones polynomial, and it is defined in a combinatorial manner, lending itself to computer calculation.  After Bar-Natan wrote up his initial calculations, a striking pattern emerged:  when plotting homological degree against quantum degree, most small crossing number knots had every nontrivial homology group lying on two adjacent ``diagonals'' of slope 2, with y-intercept of the two diagonals being the signature of the knot $\pm 1$.  Knots for which the Khovanov homology lies on only two such diagonals are called H-thin, while those with three or more diagonals are called H-thick.  The number of these diagonals on which the homology is supported is called the homological width.

Recent work of Ozsvath and Manolescu ~\cite{manolescu-2007} has grouped all but one of these small H-thin knots into a single family, called quasi-alternating knots, for which both the Khovanov and Knot Floer homologies are $\sigma$ thin, in the sense that the knots are H-thin with intercepts related to the signature.  For these knots, the reduced Khovanov homology is determined entirely by the Jones polynomial and the signature.  From this perspective, the H-thick knots are the knots with ``interesting'' Khovanov homology.

Most of the general theorems guaranteeing thickness are laid out in ~\cite{khov-patterns}. Other authors have come up with upper bounds for the homological width of a knot.  But, to actually show a knot is H-thick, one must either do an explicit calculation, or turn to the theorems of Khovanov from ~\cite{khov-patterns}, summarized below.

\begin{thm}
(Khovanov) $K_1 \# K_2$ is H-thick if and only if at least one of $K_1, K_2$ are H-thick.
\end{thm}

\begin{thm}
(Khovanov) The (usual) Jones polynomial of an H-thin knot is alternating.
\end{thm}

\begin{thm}
(Khovanov) Adequate non-alternating knots are H-thick.
\end{thm}

\noindent Some other knots for which we know the thickness come from explicit calculations of the Khovanov homology for those families.  These include some torus knots ~\cite{stosic-2006} ~\cite{turner-2006} and pretzel knots ~\cite{suzuki-2006}.  Concurrently with our work, Adam Lowrance ~\cite{lowrance-thick-rational} has determined the width of all 3-braids by showing that width does not change when ``width-preserving'' crossings are replaced by rational tangles.

\subsection{Results}
In this paper, we will demonstrate a new method for constructing H-thick knots, which can produce examples not detected by Khovanov's thickness criterion:  in particular, we will construct several infinite families of H-thick hyperbolic knots (hence prime, and not torus knots) that are not adequate and have an alternating Jones polynomial.  Instead of fully calculating the Khovanov homology, we will show that certain special homology classes, which have \emph{state cycle} representatives, persist under an operation we call \emph{quasipositive modification}.  In brief, a state cycle is when a single generator in the Khovanov chain complex, represented by an enhanced state, is a cycle; quasipositive modification is the process of ``gluing'' in a quasipositive braid in a way that is compatible with state cycles one wishes to persist.  See Figure \ref{figure-qp-schematic} for a schematic of this procedure.

\begin{figure}[hbtp]
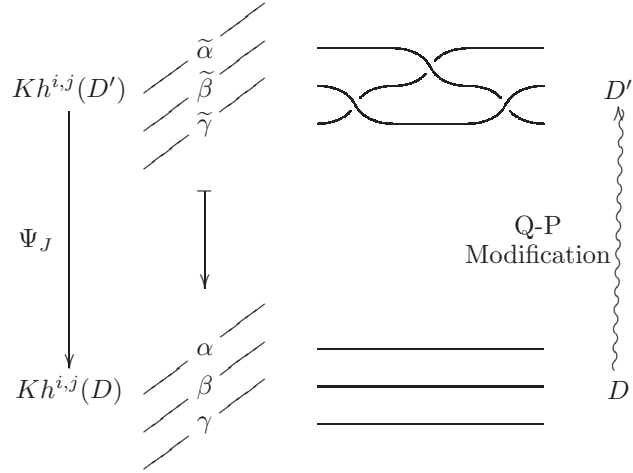

\begin{centering}
\begin{displaymath}
\xy
(-33.5,0)*++{\xy
(0,0)*{}="C1";
(8,6)*+{\gamma}="C2";
(16,12)*{}="C3";
(0,5)*{}="B1";
(8,11)*+{\beta}="B2";
(16,17)*{}="B3";
(0,10)*{}="A1";
(8,16)*+{\alpha}="A2";
(16,22)*{}="A3";
"A1"; "A2" **\dir{-}; 
"A2"; "A3" **\dir{-}; 
"B1"; "B2" **\dir{-}; 
"B2"; "B3" **\dir{-}; 
"C1"; "C2" **\dir{-}; 
"C2"; "C3" **\dir{-}; 
\endxy }="Kh1";
(21.5,0)*+{D}="Label1";
(-51.5,0)*+{\Kh(D)}="LabelKh1";
(-33.5,40)*++{\xy
(0,0)*{}="C1";
(8,6)*+{\widetilde{\gamma}}="C2";
(16,12)*{}="C3";
(0,5)*{}="B1";
(8,11)*+{\widetilde{\beta}}="B2";
(16,17)*{}="B3";
(0,10)*{}="A1";
(8,16)*+{\widetilde{\alpha}}="A2";
(16,22)*{}="A3";
"A1"; "A2" **\dir{-}; 
"A2"; "A3" **\dir{-}; 
"B1"; "B2" **\dir{-}; 
"B2"; "B3" **\dir{-}; 
"C1"; "C2" **\dir{-}; 
"C2"; "C3" **\dir{-}; 
\endxy }="Kh2";
(21.5,40)*+{D'}="Label2";
(-51.5,40)*+{\Kh(D')}="LabelKh2";
(-3.5,40)*++{\xy
(0,0)*{}="A1";
(0,5)*{}="A2";
(0,10)*{}="A3";
(10,0)*{}="B1";
(10,5)*{}="B2";
(10,10)*{}="B3";
(20,0)*{}="C1";
(20,5)*{}="C2";
(20,10)*{}="C3";
(30,0)*{}="D1";
(30,5)*{}="D2";
(30,10)*{}="D3";
\vtwist~{"A1"}{"A2"}{"B1"}{"B2"};
\vtwist~{"B2"}{"B3"}{"C2"}{"C3"};
\vcross~{"C1"}{"C2"}{"D1"}{"D2"};
"B1"; "C1" **\dir{-};
"A3"; "B3" **\dir{-};
"C3"; "D3" **\dir{-};
\endxy}="Braid2";
(-3.5,0)*++{\xy
(0,0)*{}="A1";
(0,5)*{}="A2";
(0,10)*{}="A3";
(30,0)*{}="D1";
(30,5)*{}="D2";
(30,10)*{}="D3";
"A1"; "D1" **\dir{-};
"A2"; "D2" **\dir{-};
"A3"; "D3" **\dir{-};
\endxy}="Braid1";
{\ar@{|->} "Kh2";"Kh1"};
{\ar@{->}_*+{\Psi_J} "LabelKh2";"LabelKh1"};
{\ar@{~>}^{\txt{Q-P \\Modification}} "Label1";"Label2"};
\endxy
\end{displaymath}
\end{centering}
\caption{A schematic for constructing thick knots via quasipositive modification.}
\label{figure-qp-schematic}
\end{figure}

The first main result is a prescription for what a nontrivial state cycle has to look like.  A rough statement is given by:

\begin{thm-class}
For a state cycle to represent a nontrivial homology class, the underlying state must be ``almost'' adequate, and ``most'' loops of the 1-block must be marked by $v_+$, while all 0-tracing loops must be marked by $v_-$.
\end{thm-class}

A precise statement is given in Section 4 where the theorem appears.  Our next result is that these state cycles persist under the operation of quasipositive modification:

\begin{thm-lifting-qp}
Let $D$ be an oriented diagram and $\alpha$ a nontrivial state cycle.  Suppose $D'$ is gotten from $D$ by quasipositive modification compatible with $\alpha$, and that $\Psi$ is the associated Jacobsson homomorphism from $Kh(D')$ to $Kh(D)$.  Then there exists a state cycle $\widetilde{\alpha}$ so that $\Psi(\widetilde{\alpha}) = \pm \alpha$.  If $B$ is the quasipositive braid associated to this modification, then $\widetilde{\alpha}$ is the enhanced traced state where:
\begin{itemize}
\item All crossings from $D$ are smoothed as in $\alpha$.
\item Negative crossings from $B$ are 1-smoothed, positive crossings from $B$ are 0-smoothed.
\item Every loop in $\widetilde{\alpha}$ is marked the same as $\alpha$. 
\end{itemize}
\end{thm-lifting-qp}

Furthermore, if multiple such state cycles are compatible with a quasipositive modification, the ``lifts'' of these state cycles retain the same relative grading difference; this acts as the main workhorse for constructing H-thick knots.

Quasipositive modification is also compatible with Lee's homomorphism, in the following sense:

\begin{thm-lifting-lee}
Let $D$ be an oriented diagram, and $\alpha$ a nontrivial state cycle of $D$.  Suppose $D'$ is a diagram gotten by quasipositive modification on $D$ compatible with $\alpha$, and that $\widetilde{\alpha}$ is the lift of $\alpha$.  Then $\Psi_J(\Phi_{Lee}(\widetilde{\alpha})) = \Phi_{Lee}(\Psi_J(\widetilde{\alpha})) = \Phi_{Lee}(\alpha)$.
\end{thm-lifting-lee}

This last result is useful in reducing the number of state cycles needed to construct H-thick knots:  one can find two nontrivial state cycles in distinct diagonals, and then examine Lee's spectral sequence to see that one of them is paired with a third nonzero diagonal by the Lee homomorphism.  This pairing is then preserved under quasipositive modification, so that at least three distinct diagonals will be present in the modified knot.

\subsection{Layout}
In Section 2, we go over the basics of Khovanov homology.  Section 3 introduces the notion of a state cycle, the special case when a single generator is a cycle in the Khovanov chain complex.  Section 4 gives the classification of state cycles which represent nontrivial homology classes, and Section 5 examines how state cycles interact with various maps and the Lee spectral sequence.  Section 6 carefully defines quasipositive modification and discusses how it interacts with state cycles and related cycles.  Section 7 lays out our examples of families of H-thick knots, and details other base knots where this construction works to construct H-thick knots.  Included are two infinite families of H-thick knots which cannot be detected by the Khovanov thickness criterion, and a sequence of prime links related by quasipositive modification which have increasing width.

\subsection*{Acknowledgements} The author would like to thank John Baldwin, Eric Chesebro, and Matt Hedden for helpful conversations, and Tim Cochran for encouraging him to continue with this problem.  The author also thanks Scott Morrison for sharing his recent update to JavaKh which can calculate the Khovanov homology of much larger diagrams, written for calculations in a forthcoming paper ~\cite{fgmw-spc4}.

\section{Short Review of Khovanov Homology}

Experts in the field should feel free to skip this section, with the understanding that we follow Bar-Natan's ~\cite{bn-khov-1} convention for labelling generators, and work with $\Q$ coefficients.  What follows is a review of the cube of resolutions construction of Khovanov homology.

\begin{figure}[ht!]
\centering
\includegraphics[height=50pt]{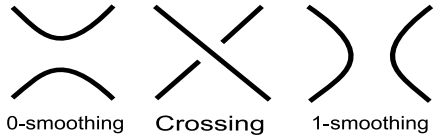}
\caption{Standard smoothing convention for a crossing}
\label{figure-smoothing}
\end{figure}

Given a diagram $ D $ for an oriented link $ L $, one can construct a state for that diagram by replacing every crossing with a choice of smoothing, per the convention in Figure \ref{figure-smoothing}.  The result will be a collection of planar loops, as shown in Figure \ref{figure-smoothing_diagram}.

\begin{figure}[ht!]
\begin{center}
\begin{displaymath}
\xymatrix{
\includegraphics{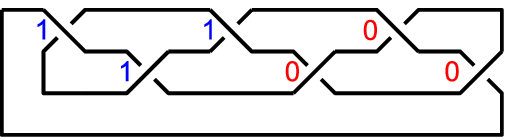} \ar@{~>}[r] & \includegraphics{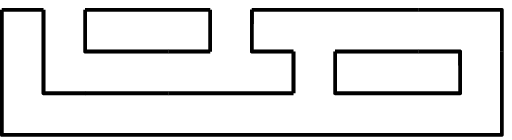} \\
}
\end{displaymath}
\end{center}
\caption{On the left, a choice of smoothings has been assigned to every crossing.  On the right, each crossing has been replaced by its chosen smoothing, giving a state for the diagram.} 
\label{figure-smoothing_diagram}
\end{figure}

By choosing an ordering for the crossings, and keeping track of the choice of 0 or 1-smoothing for each crossing, one gets a label for the state:  a string of 0's and 1's encoding the smoothing choices.  This label can also be viewed as a binary coordinate system for a hypercube whose vertices are all the possible states for the diagrams, arranged according to this coordinate system of the labels.  This hypercube is what is called the cube of resolutions for the diagram $ D $.

The chain complex for Khovanov homology is constructed by associating a graded vector space to each state, and then organizing them in some fashion that respects the structure of this cube of resolutions.  The construction is modelled on a categorification of the Kauffman state-sum formula for the Jones polynomial, with a slightly different normalization.

\subsection{Chain Groups}
Specifically, let $ V $ be the graded vector space over $\Q$ with basis $<v_+, v_->$ of grading +1, -1 respectively.  Given a state $\sigma$, let $C_\sigma$ be the graded vector space $V^{\otimes k}$, where $k$ is the number of loops in $\sigma$.  Let $h(\sigma)$, the height of the state, be the number of 1-smoothings in the label for $\sigma$.  Then, the $r^{th}$ unnormalized chain group of the diagram is defined by: 

\begin{equation}
\label{chaingroup}
C^r(D) = \bigoplus\limits_{h(\sigma) = r} C_\sigma
\end{equation}

This breakdown of generators by state lends itself to viewing generators of the chain groups as enhanced states:  states with a basis element of $V$ marked on each loop.

\subsection{Differentials}
The differential is also defined recursively at the level of these state groups, and further broken down into edge differentials, corresponding to the edges of the cube of resolutions.  An edge between two states exists if the labels for the two states differ in only a single place, where a 0 changes to a 1.  The associated edge differential is viewed as going from the state with the the 0 label to the state with the 1 label, and corresponds to the cobordism of the change from the 0 smoothing to the 1 smoothing.  This cobordism either merges two circles from the initial state, or pinches off a second circle from a circle in the initial state, and the edge differential is defined in accordance with this dichotomy, as shown in Table \ref{edge-differential}.  One labels an edge by taking the label of its originating state and substituting a * for the 0 which changes to 1 in the target of the edge.

\begin{table}[ht!]
\begin{displaymath}
\xymatrix@R=2pt{ \bigcirc \phantom{\otimes} \bigcirc \ar[r]^{\quad \mu} & \bigcirc & \qquad & \bigcirc \ar[r]^{\Delta \qquad \quad} & \bigcirc \phantom{\otimes} \bigcirc \phantom{+ \bigcirc \otimes \bigcirc}\\
v_+ \otimes v_+ \ar@{|->}[r] & v_+ & \qquad & v_+ \ar@{|->}[r] & v_+ \otimes v_- + v_- \otimes v_+ \\
v_+ \otimes v_- \ar@{|->}[r] & v_- & \qquad & v_- \ar@{|->}[r] & v_- \otimes v_- \phantom{+ v_- \otimes v_-}\\
v_- \otimes v_+ \ar@{|->}[r] & v_- \\
v_- \otimes v_- \ar@{|->}[r] & 0}
\end{displaymath}
\caption{The edge differential, $d_e$ takes one of the above forms depending on whether the associated cobordism takes two circles to one, or vice versa.  Outside of the changed part of the state, the edge differential acts as the identity.\label{edge-differential}}
\end{table}

These edge differentials are then bundled into a state differential, with a sprinkling of negative signs added in such a way that squares in the cube of resolutions anticommute.  Namely, let $|e|$ be the number of 1's which occur in the edge's label before the *, and choose $(-1)^{|e|}$ to be the sign for edge differential $d_e$.   We then combine all of the state differentials into another direct sum, to obtain the full chain differential:

\begin{equation}
\xymatrix{ d_\sigma: C_\sigma \longrightarrow C^{r+1} &
d^r: C^r \longrightarrow C^{r+1} 
\\
d_\sigma = \bigoplus \limits_{\text{edges }e} (-1)^{|e|} d_e &
d^r = \bigoplus \limits_{h(\sigma) = r} d_\sigma}
\end{equation}

For our purposes, the most important thing to remember is that edge differentials leave a state's chain group for every 0 in the state's label, and enter the state's chain group for every 1 in the state's label.  Much of the analysis of whether things are nontrivial cycles will revolve around careful analysis of the edge differentials from this perspective.

\subsection{Gradings}
To get the bigrading of $ Kh(L) $, there are some index shifts from the writhe and the cube construction that need to be addressed.  Since we will be dealing directly with chain generators, it will be sufficient to describe the bigrading of a chain generator.

So, let $ \alpha_\sigma $ be a chain generator based on state $ \sigma $.  Let $ n_+ $ and $ n_- $ be the number of positive and negative crossings of $ L $ respectively, following the usual righthanded sign convention.  Let $ v_+ (\alpha) $ and $ v_- ( \alpha ) $ denote the number of $ v_+ $ and $ v_- $ elements in the tensor for $ \alpha $.  Then, the bigrading $ (t, q) $, representing the \emph{homological grading} and \emph{quantum grading} respectively, is given by:

\begin{align}
t(\alpha) &= h(\sigma) - n_- \tag{homological grading}\\
q(\alpha) &= v_+(\alpha) - v_-(\alpha) + h(\sigma) + n_+ - 2 n_-
\tag{quantum grading}
\label{khov-bigrading}
\end{align}

All cycles constructed later will be in terms of chain generators, so this gives a concise way to calculate their bigrading.

\subsection{Diagonals}
Finally, let's review information about the diagonals of Khovanov homology.  A common way to present the Khovanov homology of a link is in a table of the following form, where the ranks of the Khovanov homology at bigrading $(i, j)$ are plotted so that homological grading goes horizontally, and quantum grading vertically.  Note that the quantum gradings shown are either all even or all odd, because the Khovanov homology of a link of $n$ components is supported only on quantum gradings $j \equiv n \pmod{2}$.  An example is given in Table \ref{table-9_41}.

\begin{table}[htbp]
\begin{center}
\setlength\arrayrulewidth{2pt}
\setlength\doublerulesep{2pt}
\begin{tabular}{|q||q|q|q|q|q|q|q|q|q|q|} \hline
 & -6 & -5 & -4 & -3 & -2 & -1 & \phantom{-}0 & \phantom{-}1 & \phantom{-}2 & \phantom{-}3\\ \hline\hline
 
\phantom{- }7  	&   	 &   	 & 		 &  	 &  	 &  	 & 		 & 		 & 		 &{\cy}1  \\ \hline
\phantom{- }5  	&   	 &   	 &  	 &  	 &  	 &		 &		 &		 &{\cy}2 &{\cb}    \\ \hline
\phantom{- }3  	&   	 &   	 & 	 	 & 		 &		 &		 &		 &{\cy}3 &{\cb}1 &    \\ \hline
\phantom{- }1  	&   	 &   	 &  	 &		 &		 &		 &{\cy}5 &{\cb}2 &  	 &    \\ \hline
\phantom{ }-1 	&   	 &		 & 		 &		 &		 &{\cy}4 &{\cb}4 & 		 &  	 &    \\ \hline
\phantom{ }-3 	&   	 &		 &		 &		 &{\cy}4 &{\cb}4 &  	 &  	 &  	 &    \\ \hline
\phantom{ }-5 	&		 &		 &		 &{\cy}3 &{\cb}4 &  	 &  	 &  	 &  	 &    \\ \hline
\phantom{ }-7 	& 		 &		 &{\cy}2 &{\cb}4 &   	 &   	 &  	 &  	 &  	 &    \\ \hline
\phantom{ }-9 	& 		 &{\cy}1 &{\cb}3 &   	 &   	 &  	 &  	 &  	 &    	 &	  \\ \hline
-11 			& {\cy}  &{\cb}2 &   	 &   	 &  	 &  	 &  	 &    	 &		 &	  \\ \hline
-13 			& {\cb}1 &		 &   	 &   	 &  	 &  	 &  	 &    	 &		 &	  \\ \hline
\end{tabular}
\end{center}
\caption{Rational Khovanov homology of $9_{41}$, an H-thin knot.  Its two diagonals are shaded.}
\label{table-9_41}
\end{table}

A pattern for the support of the Khovanov homology is that it always lies on a certain number of slope 2 diagonals (this has to do with Khovanov's Krull-Schmidt decomposition of the chain complex in ~\cite{khov-patterns}).  A diagonal of grading $\delta$ comprises all $\Kh(L)$ of the form $\delta = 2i-j$.  When a link's homology is supported on only 2 such diagonals, we say the link is \emph{H-thin}.  An example is shown in Table \ref{table-9_41}.

\begin{table}[htbp]
\begin{center}
\setlength\arrayrulewidth{2pt}
\setlength\doublerulesep{2pt}
\begin{tabular}{|q||q|q|q|q|q|q|q|} \hline
 & -4 & -3 & -2 & -1 & \phantom{-}0 & \phantom{-}1 & \phantom{-}2 \\ \hline\hline
\phantom{-}7  	&   	 &   	 & 		 &  	 &  	 &  	 &{\cy}1  \\ \hline
\phantom{-}5  	&   	 &   	 &  	 &  	 &  	 &{\cy}  &{\cb}    \\ \hline
\phantom{-}3  	&   	 &   	 & 	 	 & 		 &{\cy}1 &{\cb}1 &    \\ \hline
\phantom{-}1  	&   	 &   	 &  	 &{\cy}1 &{\cb}1 &  	 &    \\ \hline
-1 				&   	 &		 & {\cy} &{\cb}1 &{\co}1 &  	 &    \\ \hline
-3 				&   	 &{\cy}1 &{\cb}1 &  	 &  	 &  	 &    \\ \hline
-5 				& {\cy}  &{\cb}	 &  	 &  	 &  	 &  	 &    \\ \hline
-7 				& {\cb}1 &   	 &   	 &  	 &  	 &  	 &    \\ \hline
\end{tabular}
\end{center}
\caption{Rational Khovanov homology of $9_{42}$, an H-thick knot.  The three diagonals are shaded in different colors.}
\label{table-9_42-diagonals}
\end{table}

When a link's Khovanov homology is supported on 3 or more diagonals, we say the link is \emph{H-thick}.  An example is seen in Table \ref{table-9_42-diagonals}.  We say that the \emph{width} of a link's Khovanov homology is the number of diagonals on which its Khovanov homology is supported.

\section{State Cycles}

Enhanced states comprise the generators of the Khovanov chain complex.  But sometimes, a single enhanced state turns out to be a cycle representative.  These special cycles are especially convenient to manipulate, and serve as the basis for the cycles we will construct explicitly to generate our H-thick families.

\begin{defn}
We say that an enhanced state $\alpha$ is a \emph{state cycle} if the associated element of the chain complex is a cycle; namely, $d(\alpha) = 0$.  We say that a state cycle is \emph{nontrivial} if it represents a nontrivial homology class.
\label{def-state_cycle}
\end{defn}

By the definition of the Khovanov differential in terms of edge differentials leaving a state's chain group, it is a much simpler task to determine if a single enhanced state $\alpha_\sigma$ is a cycle than it is to show that it is nontrivial.  Indeed, a quick look over the edge differentials in Table \ref{edge-differential} reveals that for this to happen, every edge exiting $\sigma$ must fall in the $\mu$ case, with the respective loops marked by $ v_-$ in $ \alpha_\sigma$.  This forces every edge differential leaving $\alpha_\sigma$ to vanish; since $\alpha_\sigma$ lies in $ker(d)$, it is a cycle in $Kh(L)$.

\begin{figure}[ht!]
\centering
\includegraphics[height=50pt]{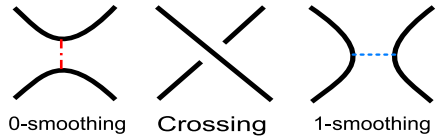}
\caption{Smoothings marked with traces of the crossings. 0-smoothings will be red, dot-and-dashed lines; 1-smoothings will be blue, dotted lines.}
\label{figure-trace}
\end{figure}

To better analyze this situation, it is helpful to record not just the smoothings of $\sigma$, but also the traces of the crossings, as shown in Figure \ref{figure-trace}. The \emph{trace} of a crossing is a shadow to show where a crossing has been smoothed to get the state $\sigma$, and represents where an edge differential either enters or exits a chain generator based on $\sigma$.  See Figure \ref{figure-trace-diff} for a schematic of this relation.  A state with all its tracings marked is said to be a \emph{traced state}; an enhanced state with all its tracing marks is said to be an \emph{enhanced trace state}, or \emph{ET state}.  See Figure \ref{figure-tracing_diagram} for an example of an ET state.

\begin{figure}[ht!]
\centering
\includegraphics{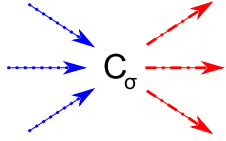}
\caption{Schematic for how traces correspond to edge differentials.  Red dot-and-dash edge differentials exit the chain group, and come from the red dot-and-dash 0-traces; blue dotted edge differentials enter the chain group, and come from the blue dotted 1-traces.}
\label{figure-trace-diff}
\end{figure}

\begin{figure}[ht!]
\begin{center}
\begin{displaymath}
\xymatrix{
\includegraphics{figures/6_3_diagram_numbered-small} \ar@{~>}[r] & \includegraphics{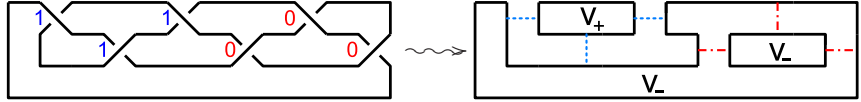} \\
}
\end{displaymath}
\end{center}
\caption{On the left, a choice of smoothings has been assigned to every crossing.  On the right, an ET state corresponding to $v_+ \otimes v_- \otimes v_-$ is shown.} 
\label{figure-tracing_diagram}
\end{figure}

Let us introduce some further terminology to discuss these traces of crossings.  We say that a trace is a \emph{mergetrace} or \emph{pinchtrace} if the trace connects two or one loop in $\sigma$, respectively.  See Figure \ref{figure-tracetypes} for some examples.  The terminology is meant to suggest that when the crossing associated to a mergetrace is changed to the opposite smoothing, the two loops joined by the mergetrace are merged together; similarly, when the crossing associated to a pinchtrace is changed to the opposite smoothing, the original loop is pinched into a pair of loops.

\begin{figure}[ht!]
\centering
\includegraphics{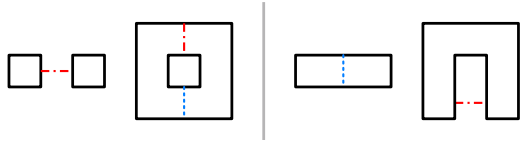}
\caption{The left half shows examples of mergetraces; the right half shows examples of pinchtraces.}
\label{figure-tracetypes}
\end{figure}

We can further differentiate traces by keeping track of which kind of smoothing they are associated to in $\sigma$.  If a mergetrace is associated to a crossing that has been 0-smoothed in $\sigma$, we say it is a \emph{0-mergetrace}, and so on.  Returning to the discussion of state cycles, the condition now becomes that $\sigma$ must be a state such that every 0-trace is a mergetrace, and that every loop touched by a 0-mergetrace must be marked by $v_-$ in $\alpha_\sigma$.

Now, let us introduce some terminology describing a state in terms of its traces.  A state is said to be \emph{0-merging} if every 0-trace in the state is a mergetrace.  Similarly, a state is said to be \emph{1-merging} if every 1-trace is a mergetrace, and so on for \emph{0-pinching} and \emph{1-pinching}.  This leads to the definition of an \emph{adequate state} as a state that is both 0-merging and 1-merging (compare to Ozawa's definition ~\cite{ozawa-state}).  

Unfortunately, the term ``adequate'' has been overloaded in the literature.  Building off the above definition, a diagram is said to be \emph{adequate} if its all-0 state and its all-1 state are both adequate as states, \emph{+ adequate} if its all-0 state is adequate as a state, and \emph{- adequate} if its all-1 state is adequate as a state.  Similarly, a link is said to be any of the above if it admits a diagram with that property.  This was the context in which Ozawa introduced the notion of an adequate state in ~\cite{ozawa-state}, where he studied properties of the surfaces associated to such states in the knot complement.

As a final bit of terminology useful in discussing ET states, we say that a loop in $\sigma$ is \emph{0-tracing} if it is touched by a 0-trace in the traced state $\sigma$.  The \emph{1-block} of $\sigma$ is the set of loops which are 1-tracing, but not 0-tracing.  It follows from our earlier discussion that for state cycles, the ET state $\alpha_\sigma$ must be 0-merging and mark every 0-tracing loop with $v_-$, but there is no condition on loops in the 1-block.  In summary, we have:

\begin{prop}
\label{zeromergingcycle}
Let $\alpha_\sigma$ be an ET state.  Then $\alpha_\sigma$ is a state cycle if and only if $\sigma$ is 0-merging and every 0-tracing loop of $\alpha_\sigma$ is marked by $v_-$.
\end{prop} 
\begin{proof}
Suppose $\alpha_\sigma$ is a state cycle.  Since it is only a single generator in the chain complex, the only way it can lie in the kernel of the differential is if every outgoing edge differential is zero.  Looking over Table \ref{edge-differential}, one sees that the only way an individual edge differential can be zero is when $d_e$ corresponds to multiplication and acts on two loops marked by $v_-$.  Since we have an outgoing edge differential for every 0-trace in the state, this means that every 0-trace must be a mergetrace and every loop touched by a 0-trace must be marked by $v_-$, as claimed.

For the opposite direction, the setup guarantees every outgoing edge differential will be zero, by the above discussion.  As the total differential is a sum of these edge differentials, $ \alpha_\sigma $ lies in $ker(d)$ and is a state cycle.
\end{proof}

\begin{example}
\label{sigma0}
Suppose $D$ is a + adequate diagram.  Let $\sigma_0$ denote the all-0 state of $D$.  By definition, $\sigma_0$ is an adequate state, and in particular 0-merging.  Let $\alpha_0$ be the ET state of $\sigma_0$ where every loop is marked by $v_-$, as shown in Figure \ref{figure-example_sigma0}.  Proposition \ref{zeromergingcycle} tells us that $\alpha_0$ is a state cycle, but because $\sigma_0$ is at the very bottom of the cube of resolutions for $D$, there are no differentials entering $\sigma_0$.  Therefore $\alpha_0$ actually represents a nontrivial homology class, one of minimal homological and quantum grading since it is at the bottom of the cube of resolutions.
\end{example}

\begin{figure}[ht!]
\begin{center}
\begin{displaymath}
\xymatrix{
\includegraphics{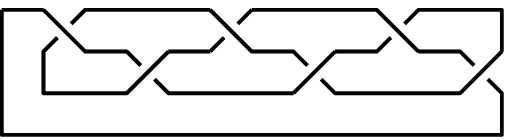} \ar@{~>}[r] & \includegraphics{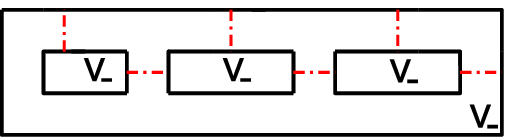} \\
}
\end{displaymath}
\end{center}
\caption{On the left is a + adequate diagram of $6_3$.  On the right, the all-0 ET state $\alpha_0$ is shown.} 
\label{figure-example_sigma0}
\end{figure}

\begin{example}
\label{sigma1}
Suppose $D$ is a - adequate diagram.  Let $\sigma_1$ denote the all-1 state of $D$.  Again, $\sigma_1$ is an adequate state, and so is both 0- and 1-merging.  Let $\alpha_1$ be the ET state of $\sigma_1$ where every loop is marked by $v_+$, as shown in Figure \ref{figure-example_sigma1}.  The set of 0-tracing loops is empty because this is the all-1 state, so Proposition \ref{zeromergingcycle} trivially holds and $\alpha_1$ is a state cycle.

However, because $\sigma_1$ lies at the top of the cube of resolutions, $C_{\sigma_1}$ is the only chain group of its height $h$ in the Khovanov chain complex for $D$.  This means that every edge differential of height $h-1$ targets $C_{\sigma_1}$, allowing us to restrict to a single edge differential for each chain group of height $h-1$.  Because every 1-trace of $\sigma_1$ is a mergetrace, every incoming edge differential must lie in the $\Delta$ half of Table \ref{edge-differential}.  By inspection, no edge differential contains a term marking $ v_+ $ on every loop of an ET state for $ \sigma_1 $.  So, no linear combination of terms can have boundary equal to $\alpha_1$.  Therefore, $\alpha_1$ does not lie in $im(d)$ for the Khovanov differential, and hence must represent a nontrivial homology class of $Kh(L)$.
\end{example}

\begin{figure}[ht!]
\begin{center}
\begin{displaymath}
\xymatrix{
\includegraphics{figures/6_3_diagram} \ar@{~>}[r] & \includegraphics{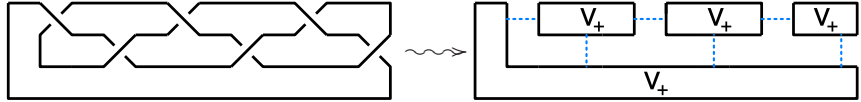} \\
}
\end{displaymath}
\end{center}
\caption{On the left is a - adequate diagram of $6_3$.  On the right, the all-1 ET state $\alpha_1$ is shown.} 
\label{figure-example_sigma1}
\end{figure}

\begin{example}
\label{olga-class}
Let $D$ be a braid diagram, and $\sigma$ be the oriented resolution, the state where every positive crossing has been 0-smoothed, and every negative crossing has been 1-smoothed.  This smoothing choice is such that the loops of the resulting state are just the strands of the braid, a concentric set of circles, with each trace going between two strands.  So, each trace is a mergetrace, and the state is adequate.  Let $ \psi $ be the ET state where every loop of $\sigma$ has been marked by $v_-$, as shown in Figure \ref{figure-example_olga}.  Proposition \ref{zeromergingcycle} tells us that $\psi $ is a cycle in Khovanov homology, but more can be said:  Plamenevskaya has shown ~\cite{olga} that $[ \psi ]$ is a transverse knot invariant, the Plamenevskaya class for $Kh(L)$.
\end{example}

\begin{figure}[ht!]
\begin{center}
\begin{displaymath}
\xymatrix{
\includegraphics{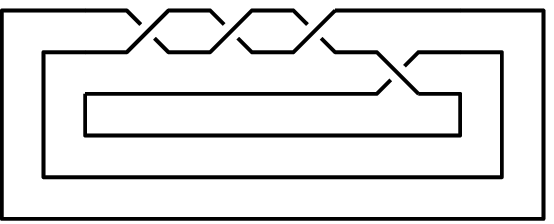} \ar@{~>}[r] & \includegraphics{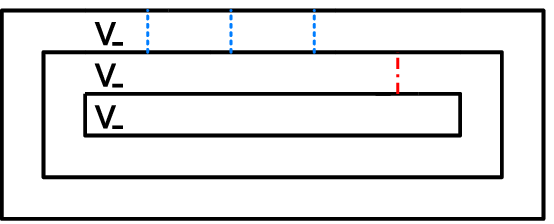} \\
}
\end{displaymath}
\end{center}
\caption{On the left is a braid diagram of the negative trefoil.  On the right, the ET state for the oriented resolution is shown, a representative for the Plamenevskaya class.} 
\label{figure-example_olga}
\end{figure}

Examples \ref{sigma0} and \ref{sigma1} are of special interest because the respective state cycles represent nontrivial homology classes.  They were a critical part of Khovanov's proof that adequate, nonalternating knots are H-thick ~\cite{khov-patterns}, and will turn out to be good models for what a general state cycle which represents a nontrivial homology class must look like.

Note that Examples \ref{sigma1} and \ref{olga-class} choose different values for their respective 1-blocks:  $v_+$ is marked on the 1-block of \ref{sigma1} while $v_-$ is marked on the 1-block of the Plamenevskaya class.  Furthermore, for all three examples, the underlying state is not simply 0-merging, but actually an adequate state.  In fact, every state cycle used in constructing the H-thick families will come from an adequate state.

\section{Classifying Nontrivial State Cycles}

In this section, we will work towards a classification of nontrivial state cycles.  The tools are a series of limits on how the 1-block can be assigned, and at the end of the section we will have a strong necessary condition for a state cycle to represent a nontrivial homology class.
  
The first step towards the classification is to place a restriction on the 1-traces of a nontrivial state cycle:

\begin{prop}
\label{adequatestatesonly}
Let $\sigma$ be a state in a diagram $D$, and $\alpha_\sigma$ an associated ET state which represents a cycle in Khovanov homology.  Suppose $\sigma$ has a 1-pinchtrace, and let $L_\sigma$ be the loop this 1-pinchtrace touches.  If $L_\sigma$ lies in the 1-block of $\sigma$, then $[ \alpha_\sigma ]$ is trivial.
\end{prop}
\begin{proof}
Let $b$ be the crossing associated to the 1-pinchtrace.  Let $\sigma_b$ be the state obtained by replacing this 1-smoothing by a 0-smoothing, and otherwise matching the smoothings of $\sigma$.  Let $L_1$ and $L_2$ be the two loops $L_\sigma$ was split into in $\sigma_b$.  Note that since $L_\sigma$ lies in the 1-block, it is only touched by 1-traces in $\sigma$:  this means that splitting it into two loops for $\sigma_b$ will not affect any of the 0-traces coming from $\sigma$, so that they remain mergetraces. 

Case 1:  Suppose $L_\sigma$ is marked by $v_+$.  Let $\beta$ be the ET state associated to $\sigma_b$ where $L_1$ and $L_2$ are marked by $v_+$, and all other loops match their marking in $\alpha_\sigma$.  Every 0-trace except that coming from $b$ will be a mergetrace between two loops marked by $v_-$ since $\alpha
_\sigma$ is a state cycle.  So, all of the edge differentials applied to $\beta$ will be 0 except for $d_b$, the one coming from $b$.  Examining Table \ref{edge-differential}, one sees that $d_b$ takes the $v_+ \otimes v_+$ on $L_1$ and $L_2$ to the $v_+$ on $L_\sigma$, and otherwise acts as the identity.  In other words, $d_b (\beta) = \alpha_\sigma$, and since all other edge differentials were zero, $d(\beta) = \pm \alpha_\sigma$.  It follows that $[ \alpha_\sigma ]$ is trivial.

Case 2: Suppose $L_\sigma$ is marked by $v_-$.  Let $\beta$ be the ET state associated to $\sigma_b$ where $L_1$ is marked by $v_+$, $L_2$ is marked by $v_-$, and every other loop matches the marking of $\alpha_\sigma$.  From the same argument as above, the only nonzero edge differential is $d_b$; examining Table \ref{edge-differential}, we see that $d_b$ again takes $\beta$ precisely to $\alpha_\sigma$.  So, $d(\beta) = \pm \alpha_\sigma$, and hence $[ \alpha_\sigma ]$ is trivial.
\end{proof}

\begin{remark}
In fact, one can loosen this condition to having one of $L_1$ or $L_2$ ``lie'' in the 1-block, in the sense that the only 0-trace touching it is the one associated to $b$.  This only matters for Case 2, where $L_\sigma$ was marked by $v_-$:  in such a situation, one marks the loop which is only 0-traced by $b$ with $v_+$, and the other by $v_-$, and the same proof holds.  When both such loops are 0-traced by traces other than $b$, one has potentially many different nonzero edge differentials, and \emph{a priori} one is not guaranteed a method of cancelling these unwanted factors.   
\end{remark}

So, every trace touching a loop in the 1-block must be a 1-mergetrace for a state cycle to represent a nontrivial homology class.  The next proposition gives some restrictions on how the loops of the 1-block can be marked:

\begin{prop}
\label{onetracepair-trivial}
Let $\sigma$ be a state in a diagram $D$, and $\alpha_\sigma$ an associated ET state which represents a cycle in Khovanov homology.  Suppose that there is some pair of loops in $\sigma$ which are only connected by 1-traces.  If those two loops are marked by $v_-$ in $\alpha_\sigma$ then $[ \alpha_\sigma ] $ is trivial.
\end{prop}
\begin{proof}
Let $b$ denote one of the 1-traces between the two loops in question, and $\sigma_b$ the state gotten by changing the smoothing for $b$ to a 0-smoothing, and otherwise retaining the smoothing choices of $\sigma$.  Since $b$ was a 1-mergetrace in $\sigma$, $b$ becomes a 0-pinchtrace in $\sigma_b$; all other 0-smoothings in $\sigma_b$ remain 0-mergetraces, since there were no 0-traces between the two loops that have been merged by changing $b$.

Let $\beta$ be the ET state for $\sigma_b$ gotten by marking the newly merged loop by $v_-$, and all other loops by their markings on $\alpha_\sigma$.  Since $\alpha_\sigma$ is a state cycle, every 0-tracing loop is a mergetrace between two loops marked by $v_-$.  This same situation holds for $\beta$ for every 0-trace except $b$, since no 0-tracing loop of $\sigma$ was altered by changing $b$.  Let the edge differential associated to $b$ be denoted $d_b$; then every edge differential on $\beta$ except $d_b$ will remain 0, as it each comes from a mergetrace between two loops marked by $v_-$.  In contrast, the edge differential $d_b$ associated to this 0-trace is nonzero:  examining Table \ref{edge-differential} one sees that $d_b$ takes the $v_-$ on the merged loop to $v_- \otimes v_-$.  So $d_b(\beta)$ is $\alpha_\sigma$, since edge differentials act as the identity outside of the changed part of the state.  Because every other edge differential was zero, $d(\beta) = \pm \alpha_\sigma$.  Therefore $[ \alpha_\sigma ]$ is trivial.
\end{proof}   

\begin{remark}
This is a generalization of Plamenevskaya's triviality condition for her transverse invariant, but the method of proof is different.  Compare to Proposition 3 of ~\cite{olga}.
\end{remark}

Even a single 0-mergetrace between the loops of Proposition \ref{onetracepair-trivial} has the chance to make $[ \alpha_\sigma ]$ nontrivial, though, depending on how the rest of the cube of resolutions is structured.  Each extra 0-mergetrace gives a new 0-pinchtrace, and another nonzero edge differential to complicate the situation.

Adjacency of a pair of $v_-$ loops is not the only problem with $v_-$ loops in the 1-block:  in fact, if there is a pair of $v_-$ loops that can be joined by a path of 1-traces in the 1-block, then the ET state is trivial: 

\begin{prop}
\label{prop-vplus-1block}
Let $\sigma$ be a state in a diagram $D$, and $\alpha_\sigma$ an associated ET state which represents a cycle in Khovanov homology.  Suppose that in $\alpha_\sigma$ there are two loops $L_\alpha$ and $L_\gamma$ marked by $v_-$, and that there exists a path of 1-traces $\{b_1, b_2, \ldots, b_n\}$ between these loops, so that every loop in this path (including $L_\alpha, L_\gamma$) lies in the 1-block.  Then $[ \alpha_\sigma ]$ is trivial.
\end{prop}
\begin{proof}
Essentially, we will construct a telescoping series of boundaries, finishing with the kind of boundary constructed in Proposition \ref{onetracepair-trivial}.  We will assume that the path chosen is minimal, in the sense that every loop between $L_\alpha$ and $L_\gamma$ in the path is marked by $v_+$ in $\alpha_\sigma$.  Furthermore, Proposition \ref{adequatestatesonly} lets us assume every 1-trace in this path is a mergetrace, as otherwise we already know that $[ \alpha_\sigma ]$ is trivial. So, let's establish some notation.

Let $\sigma_i$ be the state where crossing $b_i$ has been 0-smoothed, and all other smoothings match those of $\sigma$.  Let $d_i$ be the edge differential from $\sigma_i$ associated to the trace $b_i$.  Let $L_1$ be $L_\alpha$, and $L_i$ the loop besides $L_{i-1}$ which trace $b_{i-1}$ touches:  $L_\gamma$ then becomes $L_{n+1}$.  Furthermore, note that in $\sigma_i$, the loops $L_i$ and $L_{i+1}$ are merged, while every other loop remains the same.  Denote this merged loop in $\sigma_i$ by $M_i$.

Let $\beta_i$ be the ET state associated to $\sigma_i$ where the loops $L_1$ and  $M_i$ are marked by $v_+$ and all other loops match the marking of $\alpha_\sigma$.  Let $\delta_i$ be the ET state associated to $\sigma$ where $L_1$ and $L_i$ are marked by $v_+$, $L_{i+1}$ is marked by $v_-$, and all other loops are marked as in $\alpha_\sigma$.  See Figure \ref{figure-prop4_3-delta} for an illustration.

\begin{figure}[ht!]
\centering
\includegraphics{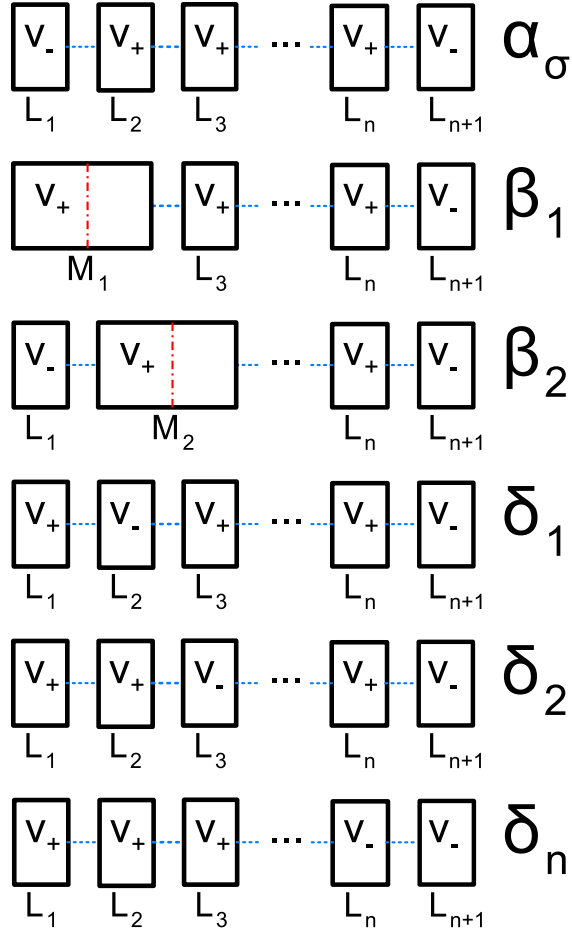}
\caption{ET states showing relevant loops for $\alpha_\sigma$, and the $\beta_i$ and $\delta_i$.}
\label{figure-prop4_3-delta}
\end{figure}

Because the only traces between each pair $L_i$ and $L_{i+1}$ are 1-traces by assumption, the only new edge differential out of $\beta_i$ will be $d_i$.  Furthermore, since $L_i, L_{i+1}$ lie in the 1-block, $d_i$ is the only edge differential out of $\beta_i$ which involves the merged loop $M_i$.  This means that all the other edge differentials remain zero:  none of the 0-tracing loops from $\alpha_\sigma$ are marked differently, and none of the 0-traces from $\alpha_\sigma$ have been changed in $\beta_i$, so that each edge differential other than $d_i$ comes from a mergetrace between two loops marked by $v_-$.

Consulting Table \ref{edge-differential}, we see that $d_i$ acting on $\beta_i$ will take the $v_+$ on loop $M_i$ to a pair of ET states, where $v_- \otimes v_+ + v_+ \otimes v_-$ suggests the marking on loops $L_i$ and $L_{i+1}$, and otherwise the markings match $\alpha_\sigma$.  In particular, writing only the tensor elements for the loops $L_i$, and marking the cobordant loops in parentheses:
\begin{align}
d_1(\beta_1) &= (v_- \otimes v_+) \otimes \cdots \otimes v_- + (v_+ \otimes v_-) \otimes \cdots \otimes v_- \\
&=\alpha_\sigma + \delta_1 \\
d_2(\beta_2) &= v_+ \otimes (v_- \otimes v_+) \otimes \cdots \otimes v_- + v_+ \otimes (v_+ \otimes v_-) \otimes \cdots \otimes v_- \\
&= \delta_1 + \delta_2 \\
d_i(\beta_i) &= \delta_{i-1} + \delta_i \\
\sum \limits_{i=1}^{n} (-1)^{i+1} d_i(\beta_i) &= \alpha_\sigma + (-1)^{n+1} \delta_n
\end{align}

Now, note that in $\delta_n$ we have an ET state where there are loops $L_n$, $L_{n+1}$ marked with $v_-$ and joined only with 1-traces.  The proof of Proposition \ref{onetracepair-trivial} gives us an explicit boundary for $\delta_n$:  let $\gamma$ be the ET state associated to $\sigma_n$ where $M_n$ is marked by $v_-$, $L_1$ is marked by $v_+$, and other loops are marked as in $\alpha_\sigma$.  Then $d(\gamma) = \pm \delta_n$.  Adding this to our sum above gives an explicit boundary for $\alpha_\sigma$, modulo some sign changes which come from the signs of the edge differentials when passing from $d_i$ to $d$.  Therefore, $[ \alpha_\sigma ]$ is trivial. 
\end{proof}

\begin{remark}
The restriction that each loop in this path must lie in the 1-block is needed to guarantee triviality.  Even a single 0-trace touching one of the loops will give a second potential nonzero edge differential for this construction that can obstruct triviality, as the following example shows:
\end{remark}

\begin{figure}[ht!]
\centering
\includegraphics{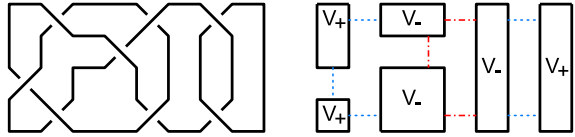}
\caption{On the left is a diagram for a link, and on the right is a particular enhanced state giving a state cycle.  Direct calculation shows that this state cycle represents a nontrivial homology class.}
\label{figure-prop4_3-1block-required}
\end{figure}

\begin{example}
Consider the link diagram and enhanced state for that diagram shown in Figure \ref{figure-prop4_3-1block-required}.  The two leftmost loops marked by $v_-$ can be joined by a path of loops joined by 1-traces, but direct calculation shows that the associated homology class is nontrivial.  So, the extra 0-traces coming off the starting and ending loops can sometimes obstruct triviality.
\end{example}

Using these propositions, we can give a fairly restricted picture of what enhanced states can represent nontrivial homology classes.  Roughly speaking, the underlying state must be ``almost'' adequate, and ``most'' loops of the 1-block must be marked by $v_+$, while all 0-tracing loops must be marked by $v_-$.  To be precise about what ``most'' loops means, we will need to introduce a number of definitions related to a graph theoretic interpretation of the traced states.

\begin{defn}
Given a state $\sigma$, the associated \emph{state graph} $\Gamma_\sigma$ is constructed by taking the loops of $\sigma$ to be vertices, and the traces to be edges.
\end{defn}

\begin{remark}
In general this is really a pseudograph, because an edge may start and end at the same vertex if the state is not adequate, and two vertices may be joined by multiple edges.  When the state is adequate, the state graph is an honest multigraph.
\end{remark}

The state graph, and subgraphs of the state graph, turns out to be a useful tool for describing conditions for triviality and nontriviality of a state cycle.  For example, consider the following definition:

\begin{defn}
We say that a state is \emph{even} if every circuit in its state graph $\Gamma_\sigma$ has even length.  Otherwise, we say that the state is \emph{odd}.
\label{def-evenstate}
\end{defn}

\begin{remark}
A pinchtrace gives a closed path of length 1 from the loop it joins to itself.  So, an even state is also an adequate state.
\end{remark}

\begin{remark}
A graph theory consequence is that if $\sigma$ is even, $\Gamma_\sigma$ is 2-colorable.  Namely, there is an assignment of a color to every vertex, using only 2 colors, so that two vertices joined by an edge have distinct colors.  One can view a 2-coloring as a sign choice for something associated to each vertex, so that one color represents ``+'' and the other represents ``-''.
\end{remark}

\begin{example}
The primary example of an even state is the Seifert state $\sigma_s$, the state gotten by choosing the smoothing for each crossing corresponding to that used in Seifert's algorithm.  $\sigma_s$ is also called the oriented resolution, because the smoothing choice is the one consistent with the orientation of the link:  positive crossings are 0-smoothed, and negative crossings are 1-smoothed, as shown in Figure \ref{figure-oriented-smoothing}.  For an example of a Seifert state and graph, see Figure \ref{figure-seifert-stategraph}.

\begin{figure}[ht!]
\centering
\includegraphics{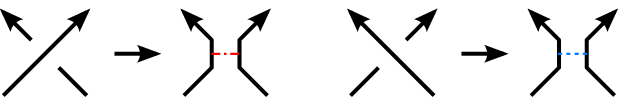}
\caption{The smoothing choice for Seifert's algorithm is the smoothing consistent with orientation.}
\label{figure-oriented-smoothing}
\end{figure}

The reason the Seifert state is even has to do with the surface $\Sigma$ associated to this state by Seifert's algorithm, formed by gluing disks to each loop of the state, and twisted bands to every trace which match the original crossing.  An odd path of traces means there is an annulus with an odd number of twists contained in the surface (gotten by following the twisted bands associated to the traces of the path), contradicting orientability of the Seifert's algorithm surface.
\label{ex-seifert-state}
\end{example}

\begin{figure}[ht!]
\centering
\includegraphics{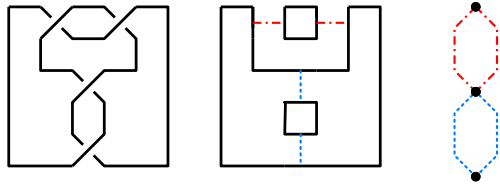}
\caption{On the left is a diagram of the figure 8.  In the middle is its Seifert state.  On the right is the state graph for this state.  By inspection this is an even state.}
\label{figure-seifert-stategraph}
\end{figure}

The fact that even states can be 2-colored turns out to be quite useful:  interpreting this 2-coloring as a sign convention can be used to construct the generators of Lee homology, which is the perspective we work with in analyzing this generator in Theorem \ref{thm-s-singleloop}.  But, for general state cycles, we will want to consider similar notions of even and oddness, restricted to the 1-block.  To describe these notions, we will examine a particular subgraph of the state graph:

\begin{defn}
Given a state $\sigma$, the associated \emph{1-block graph} $\Gamma_1$ is constructed by taking the loops of the 1-block of $\sigma$ to be vertices, and the 1-traces to be edges.
\end{defn}

In general, $\Gamma_1$ will not be connected, and its connected components turn out to be the natural setting for describing restrictions on the 1-block.

\begin{defn}
A \emph{connected component of the 1-block} refers to the set of loops from a connected component of $\Gamma_1$.  We say that a connected component of the 1-block is \emph{even} if every circuit in that component of $\Gamma_1$ has even length, and \emph{odd} otherwise.
\end{defn}

Now we are ready to state the classification theorem precisely.  Restrictions (S1) and (S2) give the precise definition of an ``almost'' adequate state, and restrictions (L2)-(L4) spell out the full obstructions to having loops in the 1-block marked by $v_-$:

\begin{thm}[State Cycle Classification]
For a state cycle $\alpha$ based on state $\sigma$ to represent a nontrivial homology class in Khovanov homology, it must satisfy the following restrictions:
\begin{itemize}
\item[(S1)] $\sigma$ must be 0-merging.
\item[(S2)] $\sigma$ can have no 1-pinchtraces touching loops in its 1-block.
\item[(L1)] 0-tracing loops of $\alpha$ must be marked by $v_-$.
\item[(L2)] No pair of loops both marked by $v_-$ can be joined by only 1-traces.
\item[(L3)] Every loop in an odd connected component of the 1-block must be marked by $v_+$.
\item[(L4)] At most one loop in an even connected component of the 1-block may be marked by $v_-$; all other loops in that component must be marked by $v_+$.
\end{itemize}
Furthermore, other than condition (L2), changing which loop is marked by $v_-$ for a fixed even component from (L4) only changes the sign of the resulting homology class.
\label{thm-statecycle-classification}
\end{thm}

\begin{proof}
Most of the tools have been assembled already.  First we will deal with the restrictions on the underlying state:

\noindent $\mathbf{(L1)}:$ Proposition \ref{zeromergingcycle} tells us that the underlying state for any state cycle must be 0-merging.

\noindent $\mathbf{(L2)}:$ Proposition \ref{adequatestatesonly} tells us that every 1-trace between loops in the 1-block must be a mergetrace.

Now, we will address the restrictions on the loop markings.

\noindent $\mathbf{(S1)}:$ Proposition \ref{zeromergingcycle} tells us that the 0-tracing loops of any state cycle must be marked by $v_-$.

\noindent $\mathbf{(S2)}:$ Proposition \ref{onetracepair-trivial} obstructs such a pair of loops marked by $v_-$ which are joined by only 1-traces.

\noindent $\mathbf{(S3)}:$ If there are two or more loops in this odd component marked by $v_-$, then there is a path of 1-traces from one to the other involving only loops in the 1-block, so Proposition \ref{prop-vplus-1block} tells us $[ \alpha ] =0$.

Suppose only one loop in the odd connected component is marked by $v_-$.  Denote this loop by $L_a$.  Since the component is odd, there is some closed path of 1-traces in that component with odd length.  Let this path be $P$.  We can construct a path $P_a$ from $L_a$ to itself of odd length by taking a path $P'$ from $L_a$ to any loop of this odd path, follow the odd path $P$, and then follow the reverse of path $P'$ back to $L_a$.  The length of $P_a$ is then twice the length of $P'$ plus the length of $P$, and so must be odd.

Let's view $P_a$ as a sequence $T_i$ of traces going from loop $L_i$ to $L_{i+1}$, of length $n$.  To each trace $T_i$, we can associate an enhanced state $\beta_i$, in the following way.  The underlying state for $\beta_i$ is $\sigma$, modified by changing the trace $T_i$ from 1 to a 0.  The loops of $\beta_i$ will have a merged loop $M_i$ corresponding to the change of the mergetrace $T_i$ to a pinchtrace, and will otherwise match the loops of $\sigma$.  So, we will label the loops of $\beta_i$ by the matching label of $\alpha$ for every loop that was unchanged, and by $v_+$ for the merged loop.  See Figure \ref{figure-thm-statecycle-betas}.

\begin{figure}[ht!]
\centering
\includegraphics{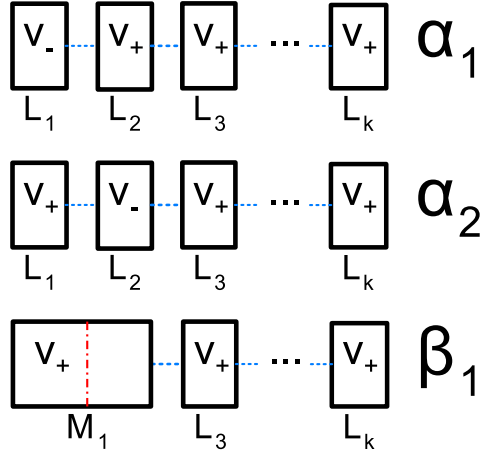}
\caption{Schematic showing the local picture for the $\alpha_i$ and $\beta_i$ in Case 1b.}
\label{figure-thm-statecycle-betas}
\end{figure}

Following the loop numbering of trace $T_i$ joining loop $L_i$ to loop $L_{i+1}$, let $\alpha_i$ denote the enhanced state where $\sigma$ is marked so that a $v_-$ is on loop $L_i$, and all other loops are marked by $v_+$, as shown in Figure \ref{figure-thm-statecycle-betas}.  By this convention, $\alpha_1 = \alpha_{n+1} = \alpha$.  By construction, $d(\beta_i) = \pm (\alpha_i + \alpha_{i+1})$:  there is only one edge differential from each state to the all-1 state, and we are in the case $\Delta(v_+) = v_- \otimes v_+ + v_+ \otimes v_-$ of Table \ref{edge-differential}.  For our purposes, we can ignore this sign associated to the edge differential during calculations and correct by the correct sign at the end.

With this setup, the claim is that $2 \alpha = \sum \limits_{i=1}^{n} (-1)^{i+1} d(\beta_i) $.  The reason is that this is a telescoping series - observe that $d(\beta_i) - d(\beta_{i+1}) = (\alpha_i + \alpha_{i+1}) - (\alpha_{i+1} + \alpha_{i+2}) = \alpha_i - \alpha_{i+2}$.  So, since the path is of odd length, the alternating sum collapses to $\alpha_1 + \alpha_{n+1} = 2 \alpha$, as claimed.  It follows that $[ \alpha ]$ is trivial.

\noindent $\mathbf{(S4)}:$ If there are two or more loops in this component marked by $v_-$, then there is a path of 1-traces from one to the other involving only loops in the 1-block, so Proposition \ref{prop-vplus-1block} tells us $[ \alpha ] =0$.

The last claim to check is that if some loop in such an even component of the 1-block is marked by $v_-$, then the choice of which loop in that component to mark $v_-$ only changes the sign of the resulting homology class.

Label the loops of this component of the 1-block by $L_1$ to $L_n$, and the traces by $T_1$ to $T_m$.  Let $\alpha_i$ be the enhanced state where loop $L_i$ is marked by $v_-$, the other $L_k$ are marked by $v_+$, and loops outside this component are marked as in $\alpha$.  Let $\sigma_j$ be the state where $\sigma$ is modified by changing trace $T_j$ from 1 to 0.  Because of restriction (L2), we know that $T_j$ is a mergetrace in $\sigma$, so $\sigma_j$ will have one loop $M_j$ which was merged by the change of smoothings, and otherwise will match the loop structure of $\sigma$.  Let $\beta_j$ be the enhanced state where $\sigma_j$ is the underlying state, $M_j$ and the other $L_k$ are all marked by $v_+$, and the other loops retain their marking from $\alpha$.

Given $i$ and $j$, there is some path of traces from $L_i$ to $L_j$ in this connected component of the 1-block.  Let this path of traces be $T'_1, \ldots, T'_p$, and trace $T'_h$ join loop $L'_h$ to loop $L'_{h+1}$.  Similarly, let $\alpha'_h$ and $\beta'_h$ be the obvious analogues in this renumbering of loops and traces.  By this setup, $d(\beta'_h) = \alpha'_h + \alpha'_{h+1}$, and we have the familiar telescoping series $\sum (-1)^{h+1} d(\beta'_h) = \alpha'_1 \pm \alpha'_p = \alpha_i \pm \alpha_j$.  So, up to sign any choice of $\alpha_i$ represents the same homology class, as claimed.
\end{proof}

\begin{remark}
While this classification theorem narrows down the candidates for nontrivial state cycles, it is a necessary condition, not a sufficient condition.  So, one still must check that a particular candidate represents a nontrivial homology class in practice.   
\end{remark}

\section{Maps acting on State Cycles}

This section examines how state cycles interact with various homomorphisms and spectral sequences.  Here we follow the paradigm that sometimes one can understand objects better by studying the maps which act on them.

The first maps we will examine are the Jacobsson homomorphisms, which give a way to relate Khovanov homology classes of two different links.  It turns out that in certain situations, Jacobsson homomorphisms can take state cycles to state cycles.  Suppose $\alpha, \beta$ are state cycles, $[ \alpha ] \neq 0$, and $\Psi(\beta) = \alpha$, for some Jacobsson homomorphism $\Psi$.  It follows that $[ \beta ] \neq 0$, since homomorphisms take 0 to 0.  In this way, Jacobsson homomorphisms can let us ``lift'' nontrivial state cycles to nontrivial state cycles, under the right setup.

The second kind of map we will examine is the Lee homomorphism $\Phi_{Lee}$.  Because of the convergence of an associated spectral sequence, in many cases $\Phi$ gives a way to organize the homology classes of a fixed diagram $D$ into pairs of the form $(\alpha, \Phi(\alpha))$, of relative bidegree (1,4).  We will see later that when the $\alpha$ of such a Lee pair is a nontrivial state cycle, there is a natural way to lift the Lee pair of one diagram to one of a related diagram.  So, we examine a special case where we can guarantee the existence of such a nontrivial Lee pair.

Finally, we conclude the section by examining a relationship between state cycles based on the Seifert state, and Rasmussen's $s$-invariant, which comes from the Lee spectral sequence.  In certain cases, we can use this information to conclude that the state cycle based on the Seifert state represents a nontrivial homology class. 

\subsection{Jacobsson Homomorphisms} The first maps we will work with are the Jacobsson homomorphisms.  For every orientable cobordism between two links, Jacobsson ~\cite{jacobsson-cob} constructed homomorphisms between the Khovanov homology of the two links and showed this construction is functorial up to sign, meaning that given two isotopic cobordisms $\Sigma_1, \Sigma_2$, the induced homomorphisms are the same up to sign.  These homomorphisms preserve the homological grading, and shift the quantum grading by the Euler characteristic of the cobordism.  In other words, a cobordism $\Sigma$ induces a $(0, \chi(\Sigma))$ homomorphism $\Psi_\Sigma$ on $\Kh(L_1)$, as shown in Table \ref{table-jacobsson-eulershift}.

\begin{table}[ht!]
\begin{displaymath}
\xymatrix@R=8pt{L_1 \ar[r]^*+{\Sigma} & L_2 \\
**[l] \Kh(L_1) \ar[r]^*+{\Psi_\Sigma} & **[r] Kh^{i, j + \chi(\Sigma)}(L_2)
}
\end{displaymath}
\caption{A cobordism $\Sigma$ induces a $(0, \chi(\Sigma))$ homomorphism $\Psi_\Sigma$ on $\Kh(L_1)$.}
\label{table-jacobsson-eulershift}
\end{table}

The construction defines maps at the chain level for each Morse and Reidemeister move, and shows that these chain maps take cycles to cycles, inducing homomorphisms.  Jacobsson's paper follows the Viro convention for marking enhanced states, so we will rewrite the homomorphisms used here in terms of the Bar-Natan generators, for convenience.  The definitions are shown in terms of ET states:  ET states not shown in the definitions will go to 0 under the homomorphisms.  The light turquoise lines in the figures indicate how arcs in the local picture of the diagram complete to loops in the state in question.

Our goal now is to analyze situations in which these homomorphisms take state cycles to state cycles, so that we can have a way of ``lifting'' nontriviality of state cycles from one diagram to another.  In this subsection we will describe a special case where this happens, which we will call \emph{positive modification}.  This construction will later be generalized to \emph{quasipositive modification} in Section \ref{q-p-mod}.

\begin{figure}[ht!]
\centering
\includegraphics{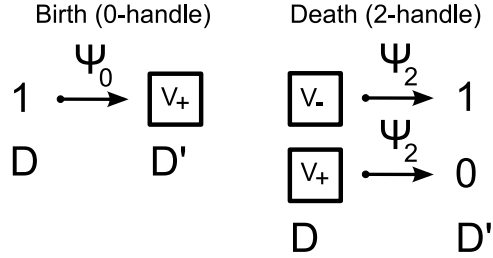}
\caption{Simplest Jacobsson homomorphisms, coming from the Morse moves for birth and death of loops.  Compare to Figure 15 on page 1226 of ~\cite{jacobsson-cob}.}
\label{figure-jacobsson-morse-0}
\end{figure}

\begin{figure}[htb!]
\centering
\includegraphics{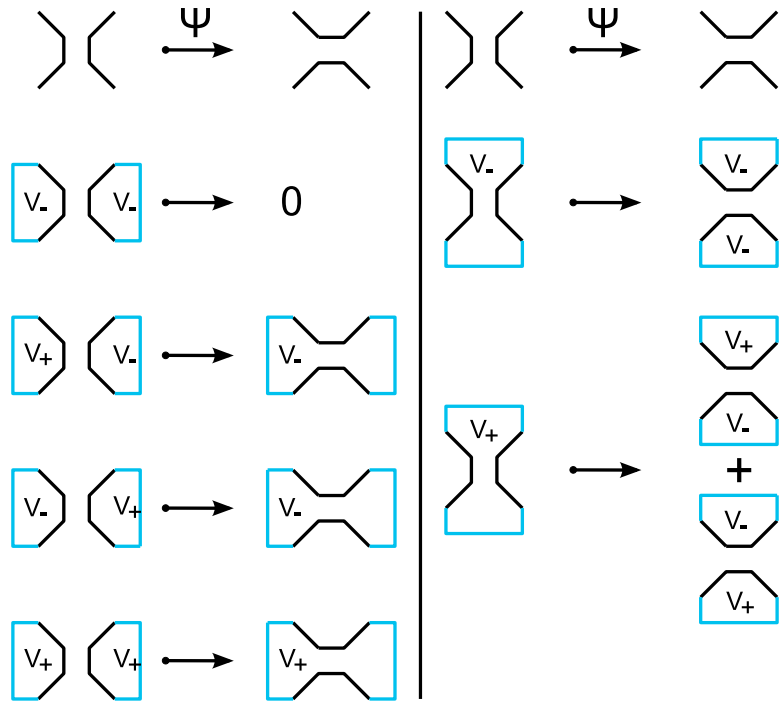}
\caption{Jacobsson homomorphisms coming from the saddle Morse moves.  Note that there are no traces here joining the loops.  Compare to Figure 3 on page 1216 of ~\cite{jacobsson-cob}.}
\label{figure-jacobsson-morse-1}
\end{figure}

\begin{figure}[htb!]
\centering
\includegraphics{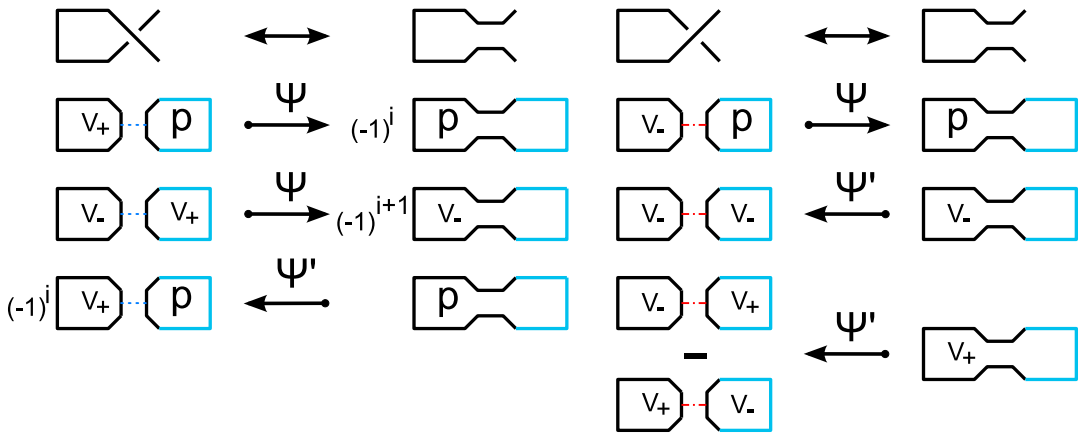}
\caption{Jacobsson homomorphisms coming from negative and positive Reidemeister I, from left to right.  Compare to Figure 16 on page 1228 of ~\cite{jacobsson-cob}.  The $p$ here indicates that either $v_-$ or $v_+$ can be marked on the loop marked by $p$.}
\label{figure-jacobsson-r1}
\end{figure}

\begin{figure}[htb!]
\centering
\includegraphics{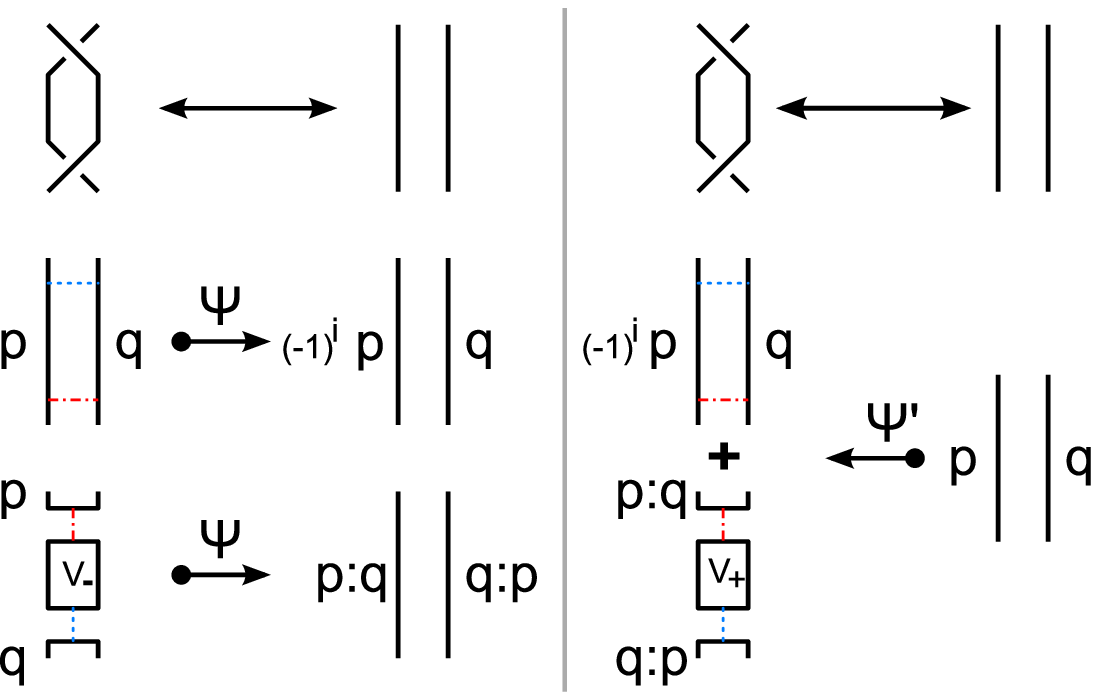}
\caption{Jacobsson homomorphisms coming from Reidemeister II.  The p:q, q:p convention indicates that locally there is a saddle cobordism from the pair of arcs $(p, q)$ to the pair of arcs $(p:q, q:p)$.  Given a choice of generators for $p$ and $q$, the values of $p:q$ and $q:p$ match those of the saddle move from Figure \ref{figure-jacobsson-morse-1} whose initial arcs match $p$ and $q$.  Compare to Figure 17 on page 1229 of ~\cite{jacobsson-cob}.}
\label{figure-jacobsson-r2}
\end{figure}


\begin{defn}
\label{positive-mod-def}
Let $D$ be an oriented diagram, and $\alpha$ an ET state representing a cycle in $Kh(D)$.  We say that an oriented diagram $D'$ is obtained from $D$ by \emph{positive modification} on $D$ compatible with $\alpha$ if there is a positive crossing $b$ in $D'$ so that 0-resolving $b$ yields $D$ as an oriented diagram and the 0-trace of $b$ connects two distinct 0-tracing loops of $\alpha$.  See Figure \ref{figure-pos-mod} for an illustration.
\end{defn}

\begin{figure}[ht!]
\centering
\includegraphics{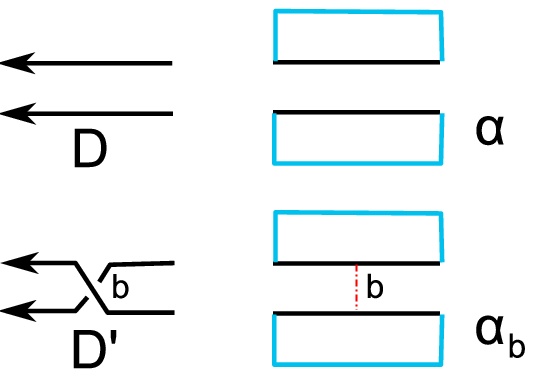}
\caption{Illustration of positive modification at the diagram and state level.}
\label{figure-pos-mod}
\end{figure}

\begin{lemma}
\label{positive-mod-lemma}
Let $D$ be a diagram for a link, and $\alpha$ a state cycle based on state $\sigma$ of $D$.  Consider a diagram $D'$ obtained from $D$ by positive modification compatible with $\alpha$.  Let $\Psi_J$ be the Jacobsson homomorphism induced by the cobordism of 0-resolving this positive crossing $b$ associated to the positive modification.  Then there exists a state cycle $\widetilde{\alpha}$ such that $\Psi_J(\widetilde{\alpha}) = \alpha$.  In particular, if $[ \alpha ]$ is nontrivial, so is $[ \widetilde{\alpha} ]$.
\end{lemma}
\begin{proof}
Let $\widetilde{\alpha}$ be the ET state in $D'$ corresponding to $\alpha$:  the loops and markings match $\alpha$, with the only difference being the added 0-mergetrace for the extra crossing $b$.  Because the two loops joined by $b$ are 0-tracing in $\alpha$, they are marked by $v_-$.  Note that since the new 0-trace for $b$ joins two loops marked by $v_-$ in $\widetilde{\alpha}$, the new edge differential for $b$ will also be zero, so that $\widetilde{\alpha}$ is also a state cycle.

\begin{figure}[ht!]
\centering
\includegraphics{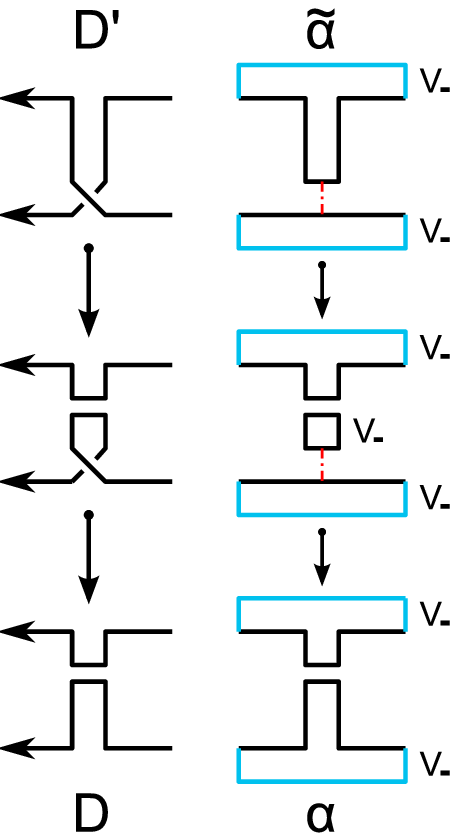}
\caption{Illustration of cobordism involved in Lemma \ref{positive-mod-lemma} at the diagram and state level.}
\label{figure-pos-mod-lem}
\end{figure}

The cobordism between $D'$ and $D$ can be broken down into adding a 1-handle, followed by undoing a positive Reidemeister I twist - see Figure \ref{figure-pos-mod-lem}.  When adding the 1-handle, we are splitting one loop into two, so the associated Jacobsson homomorphism takes the $v_-$ on that loop to $v_- \otimes v_-$.  Then, for undoing the Reidemeister one twist, the Jacobsson homomorphism takes the $v_- \otimes v_-$ pair of loops joined by the 0-trace of $b$ to a single merged loop marked by $v_-$.  So, writing this out in tensors, we have 
\begin{align}
\widetilde{\alpha} = (v_-) \otimes v_- \otimes \cdots & \stackrel{\Psi_1}{\longrightarrow} (v_- \otimes v_-) \otimes v_- \otimes \cdots \tag{1-handle} \\
v_- \otimes (v_- \otimes v_-) \otimes \cdots & \stackrel{\Psi_2}{\longrightarrow}  v_- \otimes (v_-) \otimes \cdots = \alpha \tag{RI+}
\end{align}
See Figure \ref{figure-pos-mod-lem} for an illustrated version of this.
\end{proof}

\begin{remark}
This is a generalization of Plamenevskaya's theorem about positive resolution (see Theorem 4 of ~\cite{olga}) to state cycles, and gives the key idea for ``lifting'' state cycles by Jacobsson homomorphisms.  Section \ref{q-p-mod} will extend Definition \ref{positive-mod-def} to a slightly more general kind of modification, for which we will give an analogue of Lemma \ref{positive-mod-lemma}.
\end{remark}

The Euler characteristic of the cobordism constructed in the lemma is -1, meaning that the quantum grading of $\widetilde{\alpha}$ is 1 higher than that of $\alpha$.  But, this shift is applied uniformly to every lift associated with this cobordism.

\begin{cor}
Suppose there are two state cycles, $\alpha$ and $\beta$, which are compatible with the same positive modification.  Then, the Khovanov bigrading difference between the lifts, $\widetilde{\alpha}$ and $\widetilde{\beta}$, matches the bigrading difference between $\alpha$ and $\beta$.  In particular, this means that if $\alpha$ and $\beta$ lie on two distinct diagonals in $Kh(D)$, then $\widetilde{\alpha}$ and $\widetilde{\beta}$ lie on distinct diagonals in $Kh(D')$, of the same relative separation.
\end{cor}

In theory one could try to check the nontriviality of a cycle by relating it to another known nontrivial state cycle by a more complicated series of Jacobsson homomorphisms.  In practice, this is not so easy to produce.  Even if one had a cobordism between two diagrams, the projection down of some cycle you want to investigate will probably not end up as a state cycle, limiting the option of iterating the process.  It's also very possible for the induced homomorphism to be trivial on your cycle of interest, which gives no information.

\subsection{The Lee Homomorphism}  Lee's homomorphism $\Phi_{Lee}$, which serves as a differential on $ Kh(L) $ and induces a spectral sequence ~\cite{lee}, is the next map we will analyze, to see what we can learn about state cycles.  For knots, the Lee spectral sequence converges to $\Q \oplus \Q$; by work of Rasmussen ~\cite{rasmussen-slice}, there are two homology classes in $Kh$ which survive to those two copies of $\Q$ of bigrading $(0, s \pm 1)$, where $s$ is a smooth concordance invariant to be discussed in the next subsection.  In practice, the quick convergence of this spectral sequence gives a (1,4) pairing of the other homology classes of $Kh$, which will prove of interest for state cycles.

$\Phi_{Lee}$ is defined analogously to the Khovanov differential:  first Lee defines edge differentials, as shown in Table \ref{edge-differential-lee}, then she pieces them together with signs based on the edges in exactly the same way as done in the Khovanov differential.  So, $\Phi$ is defined on the same chain groups as Khovanov homology; it satisfies $\Phi \circ \Phi = 0$, anticommutes with the Khovanov differential, and is invariant under Reidemeister moves.  This means that $\Phi + d$ gives a differential on the Khovanov chain complex $CKh$, and the resulting homology $H(CKh, d+ \Phi)$ is called the Lee homology of a link, $Kh_{Lee}(L)$.  Note also that since $\Phi$ anticommutes with $d$, $\Phi$ takes cycles of $(CKh, d)$ to cycles of $(CKh, d)$ and induces a (1,4) homomorphism on $Kh$.

\begin{table}[ht!]
\begin{displaymath}
\xymatrix@R=2pt@!C{ \bigcirc \phantom{\otimes} \bigcirc \ar[r]^{\quad \mu_\Phi} & \bigcirc & \qquad & \bigcirc \ar[r]^{\Delta_\Phi} & \bigcirc \phantom{\otimes} \bigcirc\\
v_+ \otimes v_+ \ar@{|->}[r] & 0 & \qquad & v_+ \ar@{|->}[r] & 0 \phantom{\otimes v_+} \\
v_+ \otimes v_- \ar@{|->}[r] & 0 & \qquad & v_- \ar@{|->}[r] & v_+ \otimes v_+ \\
v_- \otimes v_+ \ar@{|->}[r] & 0 \\
v_- \otimes v_- \ar@{|->}[r] & v_+}
\end{displaymath}
\caption{The edge differential for $\Phi$ shifts every nonzero result by +4 in quantum grading.\label{edge-differential-lee}}
\end{table}

In terms of the Khovanov bigrading, $\Phi$ has bidegree $(1, 4)$, while the differential $ \Phi + d$ either preserves the quantum degree or raises it by 4 for each monomial in its result.  As Rasmussen ~\cite{rasmussen-slice} observed, this means that quantum grading gives a filtration on the Khovanov chain complex which the differential $\Phi + d$ respects.  This induces a spectral sequence converging to $Kh_{Lee}$ whose $E_1$ page is given by $(Kh(L), \Phi)$.

$\Phi$ gives more information about knots whose Khovanov homology we have already calculated, particularly when this Khovanov homology has homological width less than 4.  Because the bidegree of the differentials in the Lee spectral sequence, in terms of the Khovanov bigrading,  is (1, 4i), one can often show that the spectral sequence converges at the $E_2$ page; in fact, it is conjectured that the spectral sequence always converges at the $E_2$ page.  In such a case, $\Phi$ gives a (1,4) pairing between all nontrivial Khovanov homology classes which vanish in the Lee spectral sequence.  This is not a bilinear pairing, but rather an enumeration of the homology classes by pairs of the form $(\alpha, \Phi(\alpha))$, plus the pair of homology classes which survive under the spectral sequence.  It comes from the fact that convergence of the spectral sequence at the $E_2$ page means that:
\begin{equation*}
Kh_{Lee} = \frac{\ker(\Phi)}{\im(\Phi)}
\end{equation*}

We will be interested in Lee pairs $(\alpha, \Phi(\alpha))$ where $\alpha$ is a nontrivial state cycle.  Because of the (1,4) bigrading, $\alpha$ and $\Phi(\alpha)$ lie on adjacent diagonals, so one can interpret this Lee pair as a pair of diagonals related to the state cycle $\alpha$.  Theorem \ref{lifting-lee} will tell us later that we can then lift this pair of diagonals from one diagram to another via quasipositive modification.  For now, we give a useful situation where this state cycle Lee pairing arises:

\begin{prop}
\label{lee-sigma0}
Suppose $D$ is a + adequate diagram of a knot with homological width less than 4, and that its minimum homological dimension is less than 0 (i.e. $D$ is not a positive diagram).  Then $(\alpha_0, \Phi(\alpha_0))$ is a Lee pair, where $\alpha_0$ is the ET state in which the all-0 state has every loop marked by $v_-$.
\end{prop}

\begin{proof}
Following Proposition 5.2 of ~\cite{rasmussen-slice}, the width condition on $Kh(D)$ forces the Lee spectral sequence to converge at the $E_2$ page.  The $E_1$ page of the spectral sequence is given by $(Kh(D), \Phi)$: convergence means that the homology of $Kh(D)$ under this differential $\Phi$ is $\Q \oplus \Q$ in homological dimension 0, and trivial elsewhere.

Consider the ET state $\alpha_0$, which marks every loop of the all-zero state $\sigma_0$ with $v_-$.  As observed in Example \ref{sigma0}, this represents a nontrivial homology class, in minimal quantum and homological grading.  We claim that $[ \Phi(\alpha_0) ]$ cannot be 0.

$\alpha_0$ definitely dies in the $E_2$ page of the spectral sequence, because its homological dimension is not 0 by hypothesis ($\alpha_0$ has minimal homological grading, which we assumed was less than 0).  But, because $\alpha_0$ is in the minimal bigrading and $\Phi$ is a (1,4) homomorphism, there cannot be any $\beta$ so that $\Phi(\beta) = \alpha_0$.  So, since $\alpha_0$ does not lie in the image of $\Phi$, if it lay in the kernel of $\Phi$ it would survive to the $E_2$ page, a contradiction.  It follows that $[ \Phi(\alpha_0) ]$ is nontrivial, so that $(\alpha_0, \Phi(\alpha_0))$ is a Lee pair.
\end{proof}

\begin{remark}
Note that $\Phi(\alpha_0)$ will not be a state cycle, but its form is sufficiently simple that we can see how it interacts under a certain class of Jacobsson homomorphisms later (see Theorem \ref{lifting-lee}).
\end{remark}

\subsection{The s-invariant}
In this subsection, we will examine how to exploit the relationship of certain state cycles to the $s$-invariant to conclude nontriviality of those special state cycles.  Rasmussen's $s$-invariant ~\cite{rasmussen-slice} is defined roughly to be average of the quantum gradings of the generators which survive to the $E_\infty$ page of the Lee spectral sequence.  It turns out that $s$ is a smooth concordance invariant, and it is intrisically linked to the Seifert state (the state smoothed according to the algorithm for the canonical Seifert surface of a diagram).

When a state cycle for the Seifert state has quantum grading equal to $s-1$, one can try to generalize an argument of Baldwin and Plamenevskaya ~\cite{bo-openbook} and prove that the state cycle is nontrivial.  Recall that the Seifert state is always adequate, per the argument of Example \ref{ex-seifert-state}.

\begin{prop}
Suppose $\alpha$ is a state cycle associated to the Seifert state, and its quantum grading is $s-1$.  If $\alpha$ has no 1-block, then $\alpha$ represents a nontrivial homology class.
\end{prop}
\begin{proof}
Let $\sz$ be the Seifert state, where each loop is marked by $v_- \pm v_+$ in such a fashion that loops which share a crossing alternate in the sign for $v_+$.  Such a sign choice is possible for every loop because the Seifert state is 2-colorable.  So, $\sz$ is given by the tensor $(v_- + v_+) \otimes (v_- - v_+) \otimes \cdots \otimes (v_- \pm v_+)$.  By Corollary 3.6 of Rasmussen ~\cite{rasmussen-slice}, the smallest quantum grading term of $\sz$ that is nontrivial in the Lee homology has grading $s-1$.  Following Rasmussen's notation, this means that $s(\sz) = s-1$.

When $\alpha$ has no 1-block, every loop is marked by $v_-$.  So, if you multiply out the tensor for $\sz$, one sees that $\sz = \alpha + \tau$, where every term in $\tau$ has higher quantum grading (since it has one or more $v_+$ factors).

Suppose that $[ \alpha ]$ is trivial.  Then there is some $\gamma$ so that $d \gamma = \alpha$.  Note that since the Khovanov differential $d$ preserves quantum grading, $\gamma$ also has quantum grading $s-1$.

Let $d' = d + \Phi$ be the Lee differential.  By definition, $\Phi$ shifts quantum grading up by 4.  So, $d' \gamma = d \gamma + \Phi \gamma = \alpha + \upsilon$, where the quantum grading of every term of $\upsilon$ has strictly higher quantum grading than $s-1$.

Consider $\sz - d' \gamma$.  This is homologous to $\sz$, so should have the same minimum nontrivial quantum grading as $\sz$.  But, $\sz - d' \gamma = \tau + \upsilon$, so that $s(\sz - d' \gamma) > s-1$, a contradiction.  Therefore, no such $\gamma$ exists and $\alpha$ represents a nontrivial homology class. 
\end{proof}
\begin{remark}
The argument for this case is essentially the same as that of Theorem 1.2 of Baldwin and Plamenevskaya ~\cite{bo-openbook}.  We reorganize it slightly here since it will serve as a model for other state cycles.
\end{remark}

\begin{thm}
Suppose $\alpha$ is a state cycle associated to the Seifert state, and its quantum grading is $s-1$.  If $\alpha$ has a single loop in its 1-block, and that loop is marked by $v_+$, then $\alpha$ represents a nontrivial homology class.
\label{thm-s-singleloop}
\end{thm}
\begin{proof}
This time, we will show that $\sz = \beta_0 + \beta_1 + \alpha + \tau$, where $\tau$ has quantum grading higher than $\alpha$, and both $\beta_0$ and $\beta_1$ are trivial cycles in Khovanov homology, of quantum grading $\leq s-1$.  Then, if $\alpha$ is trivial, we will again look at something homologous to $\sz$ in Lee homology, and show that it has quantum grading strictly greater than $s-1$, getting a contradiction.

First, consider $\alpha$.  Since every loop outside of the 1-block must be marked by $v_-$, up to ordering of the loops $\alpha = v_+ \otimes v_- \otimes \cdots \otimes v_-$.  Since the quantum grading of $\alpha$ is $s-1$, this means that any marking of the Seifert state with one loop marked by $v_+$, and the rest of the loops marked by $v_-$, will also have quantum grading $s-1$; a state with loops marked by only $v_-$ will have quantum grading $s-3$.

Let $\beta_0$ be the Seifert state marked by $v_-$ on every loop.  By Proposition \ref{onetracepair-trivial}, we know that $\beta_0$ is trivial because the (only) loop in the 1-block is marked by $v_-$:  it must be connected to some 0-tracing loop by only 1-traces, and that 0-tracing loop is also marked $v_-$.  So, there is some $\gamma_0$ so that $d \gamma_0 = \beta_0$, and as before, $d' \gamma_0 = \beta_0 + \upsilon_0$ for $\upsilon_0$ of quantum grading 4 higher than the grading for $\beta_0$, namely $s+1$.  In particular, $\upsilon_0$ has quantum grading higher than $s-1$.

Next, number the loops of the Seifert state other than the one in the 1-block by 1 through $n$, where the first loop is connected by a 1-smoothing to the loop marked by $v_+$ in $\alpha$.  Let $\delta_i$ be the Seifert state where loop $i$ is marked by $v_+$ and all other loops are marked by $v_-$.  Let $\epsilon_i$ denote the number of traces in a path from the loop marked by $v_+$ in $\alpha$ to loop $i$ in the Seifert state, modulo 2.  Note that this is well-defined, because the Seifert state is even, as discussed in Example \ref{ex-seifert-state}; one can also view this sign choice simply as a choice of a 2-coloring for the associated state graph.  Define $\beta_1$ by:
\begin{equation*}
\beta_1 = \sum \limits_{i=1}^{n} (-1)^{\epsilon_i} \delta_i
\end{equation*}
If we choose $\sz$ so that a $v_- + v_+$ term lies on the circle marked $v_+$ on $\alpha$, then the sign of each term of $\beta_1$ will match up with the sign of the corresponding monomial in $\sz$, since loops in $\sz$ are marked by an alternation of $v_- + v_+$ and $v_- - v_+$ by Corollary 2.5 of Rasmussen ~\cite{rasmussen-slice}.  Either such choice satisfies $s(\sz) = s-1$ by Corollary 3.6 of Rasmussen, so we are free to pick this convention for $\sz$.  By construction, we now have that $\sz = \beta_0 + \beta_1 + \alpha + \tau$, where $\tau$ are terms of quantum degree higher than $s-1$.

We claim that $\beta_1$ is a cycle.  To see this, consider first $d \delta_i$.  The only nonzero edge differentials will come from 0-traces joining the loop marked $v_+$ with an adjacent $v_-$ loop:  the edge differential of such a 0-trace will result in a state where those two loops are merged, and every loop is marked by $v_-$.  Suppose the adjacent $v_-$ loop was numbered $j$; then $\delta_j$ will also have an edge differential from that 0-trace with the same output on that edge differential.  In short, every such nonzero edge differential appears twice in $\beta_1$:  the only question is whether they appear with opposite sign in the total differential of $\beta_1$.

\begin{figure}[ht!]
\centering
\includegraphics{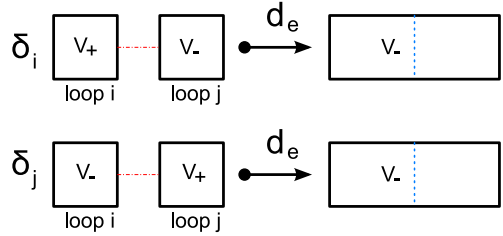}
\caption{$\delta_i$ and $\delta_j$ sharing a common 0-smoothing.}
\label{figure-beta_1}
\end{figure}

By construction, we already have that $\delta_i$ and $\delta_j$ have opposite sign in the sum for $\beta_1$, since loops $i$ and $j$ are adjacent.  But note that $\delta_i$ and $\delta_j$ have the same underlying state, so the signs of the edge differentials for the same edge (coming from the 0-smoothing between loop $i$ and loop $j$) will be the same.  So, since the two enhanced states occur with opposite sign in $\beta_1$, each such edge pair will cancel out in the total differential, as desired.  So, $\beta_1$ is a cycle.  

Now, we want to construct $\gamma_1$ so that $d \gamma_1 = \beta_1$.  Let $c$ be the crossing associated to a 1-smoothing joining loop 1 to the 1-block.
Let $\sigma_c$ be the state which matches the Seifert state, except that $c$ is now 0-smoothed.  The result is that the loop of the 1-block is merged together with loop 1 in $\sigma_c$; by the definition of the 1-block, there were no 0-smoothings between loop 1 and the 1-block loop in the Seifert state, so there is only a single 0-pinchtrace in $\sigma_c$. 

Let $\mu_1$ be $\sigma_c$ marked with a $v_+$ on the merged loop, and $v_-$ on every other loop.  Let $\mu_i$ for $i > 1$ be $\sigma_c$ with the merged loop marked $v_-$, and the other loops marked as in $\delta_i$.  In other words, $\mu_i$ has a $v_+$ on loop $i$ in $\sigma_c$ and every other loop is marked by $v_-$.

Define $\mu$ by:
\begin{equation*}
\mu = \sum \limits_{i=1}^{n} (-1)^{\epsilon_i} \mu_i
\end{equation*}
We claim that $d \mu = \beta_1 \pm \alpha$, where the sign is determined by $(-1)^{\epsilon_1}$.

First, we will assume that the crossings are ordered so that the crossings 1-smoothed in the Seifert state come last, and that the first 1-smoothing which appears in the ordering is that associated to crossing $c$.  We claim that this ordering of crossings guarantees that the signs associated to the edge differentials out of $\sigma_c$ will all be positive.

Recall that the sign for an edge differential $e$ is $(-1)^{|e|}$, where $|e|$ counts the number of crossings 1-smoothed which come before the * crossing of edge $e$ in the crossing ordering.  The set of possible * positions for an edge differential out of $\sigma_c$ corresponds to the set of 0-smoothings of the Seifert state, together with the extra crossing $c$ which is 0-smoothed in $\sigma_c$ but not the Seifert state.  Thus $|e|=0$ for every choice of *, and all edge differential signs are positive as claimed.

Consider the edge differential $d_c$ for $c$, applied to $\mu$.  Since $c$ is a 0-pinchtrace in $\sigma_c$, this edge differential takes the $\Delta$ form.  Applied to $\mu_1$, this differential returns $\alpha + \delta_1$ (this is the $v_+ \longrightarrow v_+ \otimes v_- + v_- \otimes v_+$ case); applied to $\mu_i$ for $i > 1$, this differential returns $\delta_i$ (this is the $v_- \longrightarrow v_- \otimes v_-$ case).  So, $d_c \mu = \beta_1 \pm \alpha$.
 
Now we need to show that nonzero edge differentials from crossings other than $c$ come as cancelling pairs in $\mu$, as was the case for $\beta_1$.  Other than $c$, all 0-traces are mergetraces; 0-smoothings between two $v_-$ loops will thus have zero edge differentials.  The nonzero edge differentials will come from a $v_+$ loop joining with $v_-$ loop along a 0-smoothing; each such edge differential will appear exactly twice, once from marking $v_-$ on one loop and $v_+$ on the other, and the second time from choosing the opposite markings.  Since the two loops in question are connected by a 0-trace in the Seifert state, the edge differentials appear with opposite sign in $\mu$, similar to the case we had before with edge differentials from $\beta_1$.  So, $d \mu = \beta_1 \pm \alpha$ as claimed.

Assume that $[ \alpha ]$ is trivial; then there exists $\gamma$ so that $d \gamma = \alpha$.  Similarly, $\gamma_1 = \mu \mp \gamma$ now satisfies $d \gamma_1 = \beta_1$.  Both $\alpha$ and $\beta_1$ have quantum grading $s-1$, so the same holds for $\gamma$ and $\gamma_1$.  It follows that $d' \gamma = \alpha + \upsilon$ and $d' \gamma_1 = \beta_1 + \upsilon_1$ for $\upsilon, \upsilon_1$ of quantum grading $s+3$.

Consider $\sz - d'(\gamma_0 + \gamma_1 + \gamma)$.  This is homologous to $\sz$, so $s(\sz - d'(\gamma_0 + \gamma_1 + \gamma)) = s-1$.  But, $\sz - d'(\gamma_0 + \gamma_1 + \gamma) = \tau - \upsilon_0 - \upsilon_1 - \upsilon$, so it has no monomials of degree $s-1$; therefore, $s(\sz) > s-1$, a contradiction.  It follows that $\alpha$ could not be a trivial cycle.
\end{proof}

\begin{remark}
Regarding the hypothesis on the loop in the 1-block, if that loop were marked by $v_-$ and there were more than one loop in the state, then the associated state cycle would necessarily be trivial by Proposition \ref{onetracepair-trivial}.  Spelling this out, since this is a knot, the loop of the 1-block must be connected to some 0-tracing loop.  By definition, these two loops are only connected by 1-traces, because if any 0-trace touched our loop in the 1-block, it would not be a loop of the 1-block.  The 0-tracing loop must also be marked by $v_-$, so with the 1-block loop marked by $v_-$ the hypotheses of Proposition \ref{onetracepair-trivial} apply and the state cycle must be trivial.
\end{remark}

\begin{figure}[ht!]
\centering
\includegraphics{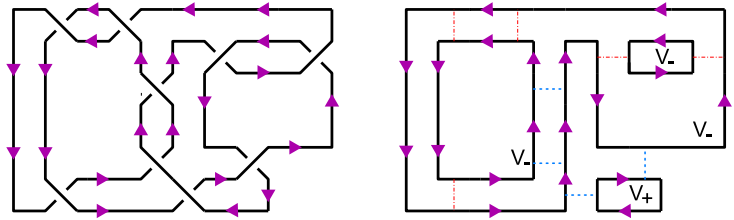}
\caption{The Seifert state cycle for $9_{42}$ has a single loop in its 1-block, and quantum grading $-1$.}
\label{figure-9_42-seifert-oet}
\end{figure}

\begin{example}
\label{example-9_42-seifert}
The $s$ invariant for $9_{42}$ is 0.  For the usual Rolfsen diagram of $9_{42}$, the state cycle associated to the Seifert state has quantum grading $-1 = s-1$ and a single loop in its 1-block, as seen in Figure \ref{figure-9_42-seifert-oet}.  So, by Theorem \ref{thm-s-singleloop}, this state cycle is nontrivial.
\end{example}

So, how about other state cycles based on the Seifert state with quantum grading $(0, s-1)$?  There is a big jump in difficulty when one adds even a second loop marked by $v_+$, but we conjecture that this relationship holds in general:

\begin{conj}
Suppose $\alpha$ is a state cycle associated to the Seifert state, and its quantum grading is $s-1$.  Then $\alpha$ represents a nontrivial homology class.
\end{conj}

\section{Quasipositive Modification}
\label{q-p-mod}

In the last section, we saw a simple procedure, positive modification, that allowed one to lift a nontrivial state cycle of one diagram to a nontrivial state cycle of another diagram.  A benefit of lifting from a state cycle to a state cycle is that the process can be iterated, allowing one to construct families of knots or links which each have nontrivial lifts of these cycles.  This section will focus on a more interesting generalization of this procedure, quasipositive modification, which still offers the ability to lift state cycles to state cycles.  

Roughly speaking, positive modification is the insertion of a positive crossing into a special place limited by the state cycle of interest.  Similarly, quasipositive modification is the process of ``gluing'' in a quasipositive braid in a way that is compatible with a state cycle.  A braid is \emph{quasipositive} if, in terms of the braid group generators $\tau_i$, it can be written in the form $\prod_k \omega_k \tau_{i_k} \omega_k^{-1}$, for some sequence of braid words $\omega_k$ and positive crossings $\tau_{i_k}$.  We will call the $\tau_{i_k}$ from such a presentation the \emph{central positive crossings} of the braid word, to distinguish them from the positive crossings that might occur in the $\omega_i$.

\begin{figure}[ht!]
\centering
\includegraphics{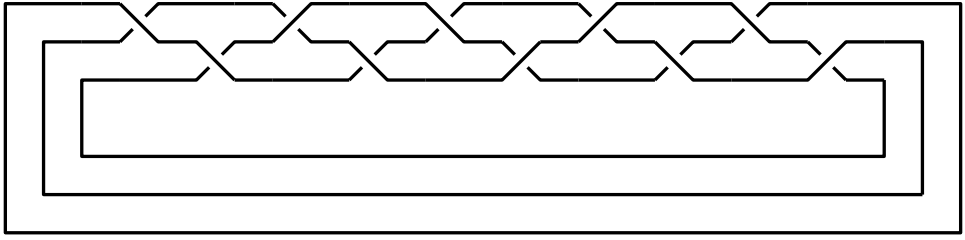}
\caption{Closure of the quasipositive braid for the mirror image of $8_{20}$.}
\label{figure-8_20}
\end{figure}

\begin{example}
\label{example-q-p-8_20}
The mirror image of the knot $8_{20}$ has a diagram of the form $ (\tau_1 \tau_2 \tau_1^{-1} \underline{\tau_2} \tau_1 \tau_2^{-1} \tau_1^{-1})( \tau_2 \underline{\tau_1} \tau_2^{-1}) $.  Here, the breakdown into each factor $\omega_i \tau_{u_k} \omega_i^{-1}$ is marked by parenthesis, with the central positive crossings underlined.  See Figure \ref{figure-8_20}.
\end{example}

For further examples, Baader has classified which prime knots up to 10 crossings can be realized as quasipositive braids, and Appendix A of ~\cite{baader-thesis} has a full list of knots up to 10 crossings which are positive and quasipositive, with quasipositive braid words for those which are honestly quasipositive.  Be aware that this table does not distinguish mirror images - usually a Jones polynomial calculation will establish that the braid in question is for the mirror image of the knot, as is the case here for $8_{20}$.

When we glue in a quasipositive braid, orientation will be an issue, as we want the resulting diagram to keep the orientations of the crossings of the quasipositive braid.  We will also want to place restrictions based on the state cycle we wish to lift.  This leads to the next iteration of state-based definitions:

\begin{figure}[ht!]
\centering
\includegraphics{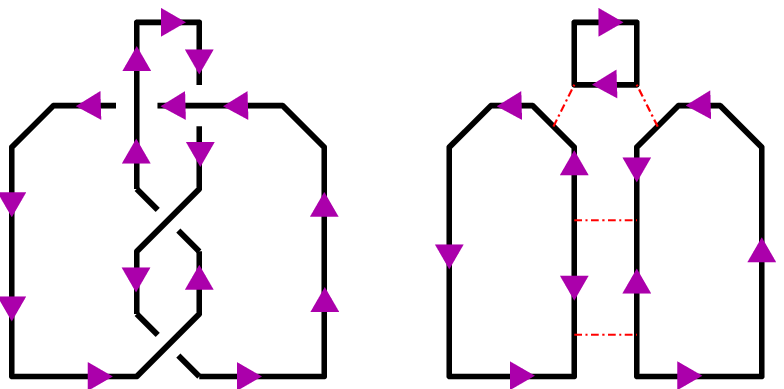}
\caption{Oriented diagram for the figure 8, and the OT state for its all zero state.}
\label{figure-ot-figure8}
\end{figure}

\begin{defn}
Let $D$ be an oriented link diagram.  An \emph{Oriented Traced State}, or \emph{OT State}, is a traced state for $D$, for which every arc between traces has been marked by the orientation of that arc in the original diagram.  Similarly, an \emph{Oriented Enhanced Traced State}, or \emph{OET State}, is an ET state marked by this orientation information.  If $\alpha$ is an ET state, denote the associated OET state by $<\alpha>$.
\end{defn}

As shown in Figure \ref{figure-ot-figure8}, the orientations of the subarcs of a loop in an OT state will not in general give a consistent orientation for this loop.  But, these oriented arcs will help specify the region we will be gluing in quasipositive braids, for quasipositive modification.

\begin{defn}
Let $D$ be an oriented link diagram, and $<\sigma>$ an OET state for that diagram representing a nontrivial homology class.  A collection of oriented arcs from $<\sigma>$,  $A_\sigma$, is said to be \emph{braid-parallel with respect to $<\sigma>$} if:
\begin{enumerate}
\item Each of the arcs comes from separate loops in $\sigma$.
\item Each of the arcs is marked by $v_-$ in $\sigma$. 
\item The arcs can be joined by a line transverse to each which touches only the arcs of $A_\sigma$ and meets each arc in the same orientation.
\end{enumerate}
\end{defn}

Roughly speaking, the braid-parallel arcs are what we will replace by a quasipositive braid, for quasipositive modification.  The matching orientation of each arc agrees with how each braid strand is oriented in a braid, so that when replacing the braid-parallel arcs by the braid, positive and negative crossings in the braid will remain positive and negative in the new diagram.  Alternatively, one can replace the transverse line by a ``box'' around the arcs, so that inside the box, the arcs look like the trivial braid on $n$ strands.  See Figure \ref{figure-ot-trefoil-boxed}.

\begin{figure}[ht!]
\centering
\includegraphics{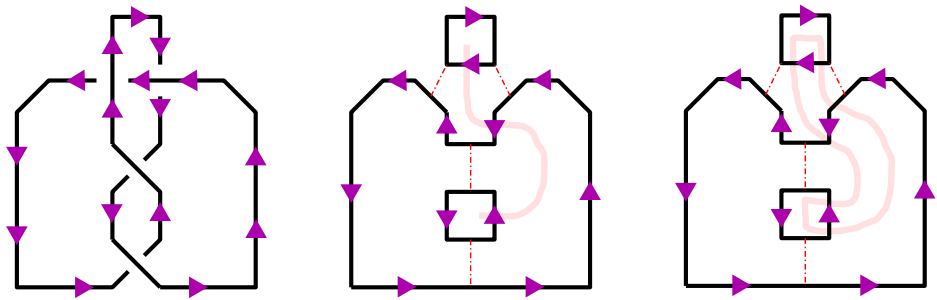}
\caption{Oriented diagram for a figure-8 look-alike version of the positive trefoil, OT states for the all zero state, and a collection of 3 braid-parallel arcs marked by a transverse line, then a box.}
\label{figure-ot-trefoil-boxed}
\end{figure}

\begin{defn}
\label{def-q-p-mod}
Given oriented diagrams $D$ and $D'$, and an OET state $< \sigma >$ of $D$, we say that $D'$ is a quasipositive modification on $D$ compatible with $<\sigma>$ if there is a collection of arcs of $D$ braid-parallel with respect to $<\sigma>$ and a quasipositive braid $B$ on as many strands so that replacing the arcs in $D$ with the quasipositive braid results in $D'$.
\end{defn}

\begin{figure}[ht!]
\centering
\includegraphics{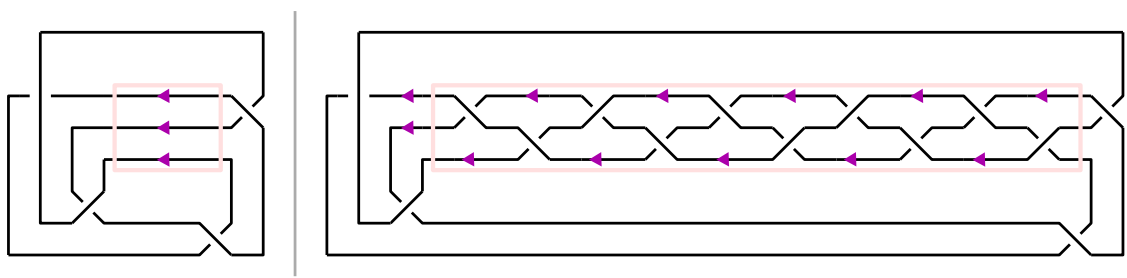}
\caption{On the left, the figure-8 look-alike diagram of the positive trefoil is isotoped so that the braid-parallel arcs shown in Figure \ref{figure-ot-trefoil-boxed} are actually parallel.  On the right is the quasipositive modification of this diagram by the mirror of $8_{20}$.}
\label{figure-q-p-trefoil}
\end{figure}

\begin{example}
\label{example-q-p-trefoil}
Let $D$ be the figure-8 look-alike diagram of the positive trefoil (same diagram as Figure \ref{figure-ot-trefoil-boxed}).  This is a + adequate diagram, so let $\alpha_0$ be the ET state for the all-zero state where each loop is marked by $v_-$.  Figure \ref{figure-ot-trefoil-boxed} has a collection of 3 arcs braid-parallel with respect to $<\alpha>$, so let $D'$ be the diagram obtained by replacing those arcs with the quasipositive braid for the mirror of $8_{20}$.  Then $D'$ is a quasipositive modification of $D$ compatible with $\alpha$.  See Figure \ref{figure-q-p-trefoil}.
\end{example}

Alternatively, one can view quasipositive modification as a procedure where one does a series of Reidemeister II moves followed by adding positive crossings, in such a way that removing the positive crossings yields a diagram that can be simplified, using Reidemeister II only, to the original diagram.  The compatibility condition with a state cycle then limits where one can do the Reidemeister moves and add positive crossings; from this perspective, quasipositive modification is a series of positive modifications on a Reidemeister II modified diagram.

The first main benefit of quasipositive modification with respect to a nontrivial state cycle is that the state cycle can be lifted to a nontrivial state cycle for the new diagram, via an associated Jacobsson homomorphism.

\begin{thm}
\label{lifting-q-p}
Let $D$ be an oriented diagram and $<\alpha>$ an OET state for a nontrivial state cycle.  Suppose $D'$ is gotten from $D$ by quasipositive modification compatible with $<\alpha>$, and that $\Psi$ is the associated Jacobsson homomorphism from $Kh(D')$ to $Kh(D)$.  Then there exists a state cycle $\widetilde{\alpha}$ so that $\Psi(\widetilde{\alpha}) = \pm \alpha$.  If $B$ is the quasipositive braid associated to this modification, then $\widetilde{\alpha}$ is the ET state where:
\begin{itemize}
\item All crossings from $D$ are smoothed as in $\alpha$.
\item Negative crossings from $B$ are 1-smoothed, positive crossings from $B$ are 0-smoothed.
\item Every loop in $\widetilde{\alpha}$ is marked the same as $\alpha$. 
\end{itemize}
\end{thm}
\begin{proof}
The basic strategy here is to first ``remove'' the central positive crossings, via Lemma \ref{positive-mod-lemma}, then use Reidemeister II moves to get rid of the remaining conjugate pairs from the braid word.  At each stage, the loops, and the markings on the loops, should remain the same:  the only difference is that there are less traces from crossings, since we are progressively simplifying the diagram.

\begin{figure}[ht!]
\centering
\includegraphics{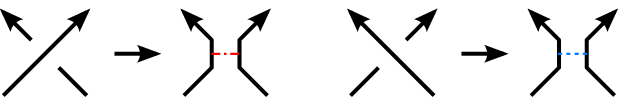}
\caption{Illustration of how the oriented resolution looks at the ET state level for positive, negative crossings respectively.}
\label{figure-oriented-resolution}
\end{figure}

The first thing to check is that $\widetilde{\alpha}$ makes sense:  it needs to have the same underlying loop structure as $\alpha$, with additional traces coming from all the extra crossings in $D'$.  Because of the orientation condition on the collection of braid-parallel arcs, the braid segment from the quasipositive modification retains all the crossing signs of the original braid:  positive crossings remain positive, negative crossings remain negative.  Checking the cases shown in Figure \ref{figure-oriented-resolution}, we see that the smoothing choice given for $\widetilde{\alpha}$ locally results in the original braid-parallel arcs, with extra traces coming from the additional crossings.  So, the loops of $\widetilde{\alpha}$ are exactly the same as for $\alpha$, and the marking choice gives us the entirety of an ET state.

All of the new crossings in $D'$ go between the braid-parallel arcs, and the choice of resolutions has the traces of those crossings going between two arcs of the collection.  Since each arc comes from its own loop in $\alpha$ and $\widetilde{\alpha}$ has the same loop structure, each of the new traces are mergetraces in $\widetilde{\alpha}$ between loops marked by $v_-$.  In particular, the new 0-traces are mergetraces between two loops marked by $v_-$, and the old 0-traces remain mergetraces between loops marked by $v_-$, since the underlying loop structure of $\widetilde{\alpha}$ is the same as that of $\alpha$ and none of the original traces have been altered.  It follows from Proposition \ref{zeromergingcycle} that $\widetilde{\alpha}$ represents a state cycle.

At each step, whether we 0-resolve a central positive crossing, or perform a Reidemeister II move, the loop structure will remain the same:  we will just be removing traces going from ET state to ET state.  So, as long as the markings on the loops in question do not change, we will have a sequence of Jacobsson homomorphisms going from $\widetilde{\alpha}$ to $\pm \alpha$ as desired.  The Positive Modification lemma (\ref{positive-mod-lemma}) shows this holds for resolving the central positive crossings, so it suffices to check this situation holds for the Reidemeister II moves.

For the Reidemeister II moves, at each step we have a negative and positive crossing adjacent in the diagram.  At the state level, these have been 1-smoothed and 0-smoothed, and the respective traces are both mergetraces going between the same two arcs.  This puts us in the situation of the top left part of Figure \ref{figure-jacobsson-r2}, where we see that the markings of the two arcs remain the same after the Jacobsson homomorphism is applied, with a possible sign change on the resulting state cycle.  So, the Reidemeister II homomorphisms act as claimed, and the resulting image of the last homomorphism will be $\pm \alpha$.
\end{proof}

\begin{remark}
This is a generalization of Lemma \ref{positive-mod-lemma}; the corresponding result in Plamenevskaya's case is that the contact invariant of a quasipositive braid is nontrivial (Corollary 1 of ~\cite{olga}).  A feature of our construction is that a quasipositive modification can be compatible with multiple nontrivial state cycles, a fact we will exploit in Section \ref{case-9_42} to construct families of H-thick knots.  Additionally, since it takes nontrivial state cycles to nontrivial state cycles, it is a process that can be iterated.
\end{remark}

Suppose the braid associated to the quasipositive modification has $k$ central positive crossings.  Then the cobordism associated to this modification has Euler characteristic $-k$, so that the quantum grading of a lift is $k$ higher than the quantum grading of the original state cycle.  As was the case with positive modification, this shift is the same for every state cycle compatible with the quasipositive modification, because the Jacobsson homomorphism induced by cobordism $\Sigma$ is a $(0, \chi(\Sigma))$ map:

\begin{cor}
Suppose $<\alpha>$ and $<\beta>$ are both compatible with the same quasipositive modification.  Then the relative grading difference of $\alpha$ and $\beta$ also holds for the lifts $\widetilde{\alpha}$ and $\widetilde{\beta}$.  In particular, if $\alpha$ and $\beta$ lie on distinct diagonals of $Kh(D)$, then $\widetilde{\alpha}$ and $\widetilde{\beta}$ lie on distinct diagonals of $Kh(D')$.
\end{cor}

In other words, if one can find nontrivial state cycles in $D$ on $n$ different diagonals which are compatible with a quasipositive modification, the lifts guarantee that $D'$ will also have homological width at least $n$.  Unfortunately, without more tools it can be difficult to find that many nontrivial state cycles. 

However, another feature of quasipositive modification is that the associated Jacobsson homomorphism also lifts state cycle Lee pairs:

\begin{displaymath}
\xymatrix{
 & \Phi(\widetilde{\alpha}) \ar[dd]^{\Psi_J} \\
 \widetilde{\alpha} \ar[dd]_{\Psi_J} \ar[ur]^{\Phi}& \\
 & \Phi(\alpha) \\
 \alpha \ar[ur]^{\Phi} & \\
}
\end{displaymath}

\begin{thm}
\label{lifting-lee}
Let $D$ be an oriented diagram, and $<\alpha>$ an OET state for a nontrivial state cycle of $D$.  Suppose $D'$ is a diagram gotten by quasipositive modification on $D$ compatible with $\alpha$, and that $\widetilde{\alpha}$ is the lift of $\alpha$.  Then $\Psi_J(\Phi_{Lee}(\widetilde{\alpha})) = \Phi_{Lee}(\Psi_J(\widetilde{\alpha})) = \Phi_{Lee}(\alpha)$.
\end{thm}

\begin{proof}
By Theorem \ref{lifting-q-p} we have a state cycle lift of $\alpha$, $\widetilde{\alpha}$.  The first thing to consider is how $\Phi$ acts on these two state cycles.  In both cases, every 0-trace is a mergetrace between two loops marked by $v_-$, so the edge maps from $\Phi$ will be of the form:
\begin{align*}
\mu_\Phi: \bigcirc \phantom{\otimes} \bigcirc & \longrightarrow \bigcirc \\
v_- \otimes v_- &\longrightarrow v_+
\end{align*}
Each of these edge maps targets a different state, so $\Phi(\alpha)$ and $\Phi(\widetilde{\alpha})$ will be a sum of ET states, each of which has a merged loop marked by $v_+$ corresponding to the associated 0-mergetrace.

Furthermore, regardless of whether the Lee spectral sequence converges at $E_2$ in $D'$, $\Phi(\widetilde{\alpha})$ will still be a cycle, since $\Phi$ is a homomorphism on $Kh(D')$.  We want to show that this maps down to $\pm \Phi(\alpha)$ under the Jacobsson homomorphism for the quasipositive modification.  This can be broken down into three cases, by the kinds of terms in $\Phi(\widetilde{\alpha})$:  image states coming from 0-traces of central positive crossings, image states coming from 0-traces of other positive crossings in the braid, and image states coming from 0-traces of $\alpha$.  The claim is that under the Jacobsson homomorphisms, the first two kinds of terms vanish, and the last terms survive to the equivalent terms of $\Phi(\alpha)$.

\begin{figure}[ht!]
\centering
\includegraphics{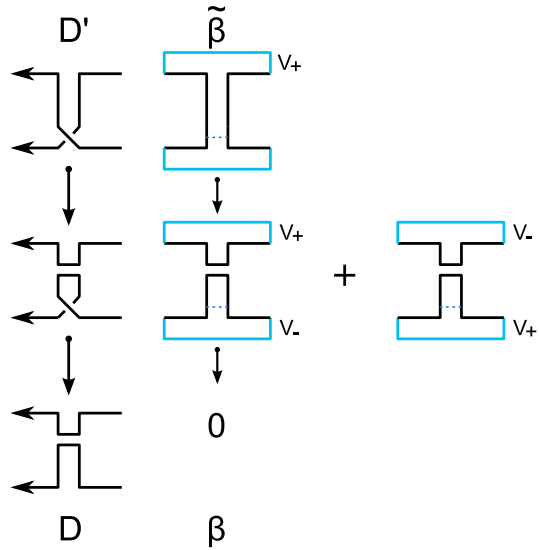}
\caption{ET states and diagrams showing how a term coming from a central positive crossing vanishes under the Jacobsson homomorphism associated to 0-resolving that crossing.}
\label{figure-lee-case1}
\end{figure}

Case 1:  Suppose $\widetilde{\beta} \in \Kh(D')$ is the $\Phi$ edge image of a 0-trace of a central positive crossing.  Jacobsson homomorphisms commute up to sign, so it suffices to show $\widetilde{\beta}$ vanishes under the 0-resolution of this central positive crossing.  As shown in Figure \ref{figure-lee-case1}, this ET state survives the homomorphism for the 1-handle, but vanishes under the homomorphism for positive Reidemeister I.  Consulting Figure \ref{figure-jacobsson-r1}, one sees that each state where the 1-trace of that positive crossing occurs is not listed, which means that it is sent to 0 under this Reidemeister I homomorphism.

\begin{figure}[ht!]
\centering
\includegraphics{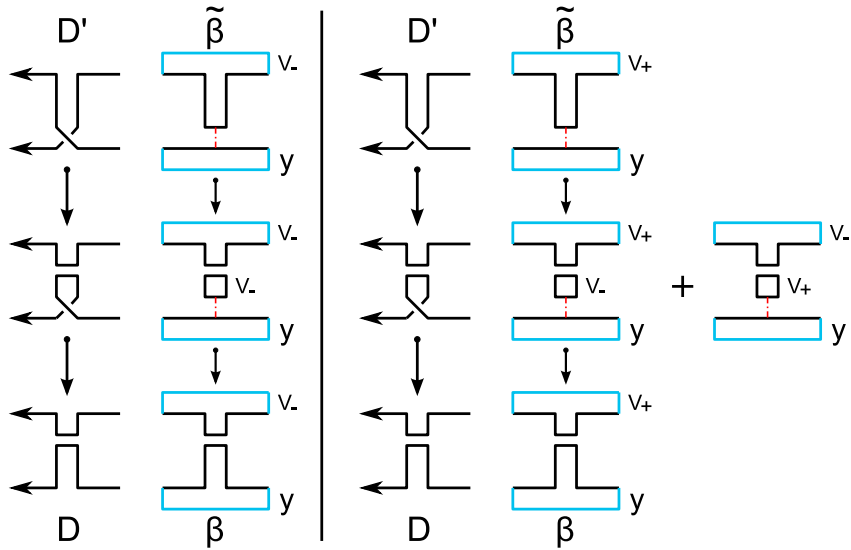}
\caption{ET states and diagrams showing how ET states which have 0-traces for a positive crossing behave under 0-resolution of that positive crossing.}
\label{figure-lee-case2a}
\end{figure}

Case 2:  Suppose $\widetilde{\beta} \in \Kh(D')$ is the $\Phi$ edge image of a 0-trace of a noncentral positive crossing from the braid.  Figure \ref{figure-lee-case2a} shows that $\widetilde{\beta}$ is ``preserved'' under 0-resolution of every central positive crossing, in the sense that the underlying loop markings and state remain the same under each such Jacobsson homomorphism.  To see that it vanishes under the Reidemeister II move which cancels this 1-smoothed crossing, consider the lefthand side of Figure \ref{figure-jacobsson-r2}.  $\widetilde{\beta}$ will have 1-smoothed both the positive and negative crossing of this Reidemeister II pair, but the only nonzero images under the Jacobsson homomorphism come from 0-smoothing one of the two crossings and 1-smoothing the other.  So, $\widetilde{\beta}$ vanishes under this Jacobsson homomorphism as claimed. 

Case 3:  Suppose $\widetilde{\beta} \in \Kh(D')$ is the $\Phi$ edge image of a 0-trace not coming from the braid.  As in Case 2, $\widetilde{\beta}$ is preserved under the 0-resolution of every central positive crossing.  As for the Reidemeister II moves, since the 1-smoothed crossing does not occur on any of the Reidemeister II pairs, we fall into the top lefthand case of Figure \ref{figure-jacobsson-r2}, which preserves the markings of the loops involved each time the homomorphism is applied, up to sign.  So, the net result is $\pm \Phi_e(\alpha)$, where $e$ is the associated edge.

The last thing one must deal with is whether there are any global problems with all the sign changes that are taking place in the Jacobsson homomorphisms.  But, each of the homomorphisms will apply the same sign change to all of the surviving terms, so that we really do have $\Psi(\Phi(\widetilde{\alpha})) = \pm \Phi(\alpha)$ as claimed. 
\end{proof}

\begin{example}
\label{example-q-p-lee}
Suppose $D$ is a quasipositive diagram for a knot of homological width less than 4, with at least one negative crossing.  As seen in Proposition \ref{lee-sigma0}, if $\alpha_0$ is the usual + adequate all-zero cycle, then $\alpha_0$ and $\Phi(\alpha_0)$ are nontrivial homology classes.  Therefore, $\widetilde{\alpha_0}$ and $\Phi(\widetilde{\alpha_0})$ are nontrivial homology classes in $D'$.  This theorem thus provides a method of lifting the two adjacent diagonals these homology classes lie on, so that there is a diagonal that lies above the diagonal for $\widetilde{\alpha_0}$.

If $D'$ remained + adequate, this would follow immediately from Khovanov's analysis of the Krull-Schmidt decomposition of the Khovanov chain complex  in ~\cite{khov-patterns}, but $D'$ as a diagram does not remain + adequate under quasipositive modification.  Without knowing that $\widetilde{\alpha_0}$ is minimal in $Kh(D')$, one cannot \emph{a priori} tell whether the second diagonal lies above or below the diagonal for $\widetilde{\alpha_0}$. 

Furthermore, this process of lifting such Lee pairs can be iterated even when the resulting $D'$ is not + adequate:  for nontriviality of the lifts, all we require is that $[ \alpha ]$ and $[ \Phi(\alpha) ]$ are both nontrivial homology classes.
\end{example}

\section{Families of H-thick knots via Quasipositive Modification}
\label{case-9_42}

In this section, we examine the state cycles of a diagram of the knot $9_{42}$ and exhibit how to perform several compatible quasipositive modifications.  In this way, we construct several families of H-thick knots which cannot be detected by Khovanov's thickness criterion, and a sequence of prime knots and links related by quasipositive modification for which width is increasing.  Finally, we discuss other potential sources of base knots for quasipositive modification.

\subsection{The Base Knot $9_{42}$}

$9_{42}$ is an H-thick knot, and this subsection will show three homology classes, one in each of its diagonals, which we can lift by quasipositive modification.  Table \ref{table-9_42} lists its rational Khovanov homology according to the Knot Atlas ~\cite{bn-atlas}, with the three homology classes we will lift marked by shading.

\begin{table}[htbp]
\begin{center}
\setlength\arrayrulewidth{2pt}
\setlength\doublerulesep{2pt}
\begin{tabular}{|q||q|q|q|q|q|q|q|} \hline
 & -4 & -3 & -2 & -1 & \phantom{-}0 & \phantom{-}1 & \phantom{-}2 \\ \hline\hline
\phantom{-}7  	&   	 &   	 & 		 &  	 &  	 &  	 & 	1  \\ \hline
\phantom{-}5  	&   	 &   	 &  	 &  	 &  	 &  	 &     \\ \hline
\phantom{-}3  	&   	 &   	 & 	 	 & 		 & 	   1 &		1 &    \\ \hline
\phantom{-}1  	&   	 &   	 &  	 &	   1 &		1 &  	 &    \\ \hline
-1 				&   	 &		 & 		 &	   1 &{\co}1 &  	 &    \\ \hline
-3 				&   	 &{\co}1 &		1 &  	 &  	 &  	 &    \\ \hline
-5 				& 	     &   	 &  	 &  	 &  	 &  	 &    \\ \hline
-7 				& {\co}1 &   	 &   	 &  	 &  	 &  	 &    \\ \hline
\end{tabular}
\end{center}
\caption{Rational Khovanov homology of $9_{42}$.  Homology classes which will be lifted under quasipositive modification have been shaded.}\label{table-9_42}
\end{table}

The first thing to note is that the ``usual'' diagram from Bar-Natan's Knot Atlas ~\cite{bn-atlas} for $9_{42}$ is + adequate; an isotoped form of this diagram is shown in Figure \ref{figure-9_42-sigma0}.  This gives us a state cycle representative, $\alpha_0$, for the minimal quantum and homological graded entry of its rational Khovanov homology, (-4, -7).  Following Example \ref{example-q-p-lee}, this gives us a second nontrivial, liftable homology class from its Lee pair at bigrading (-3,-3), since $9_{42}$ has width 3.

\begin{figure}[ht!]
\centering
\includegraphics{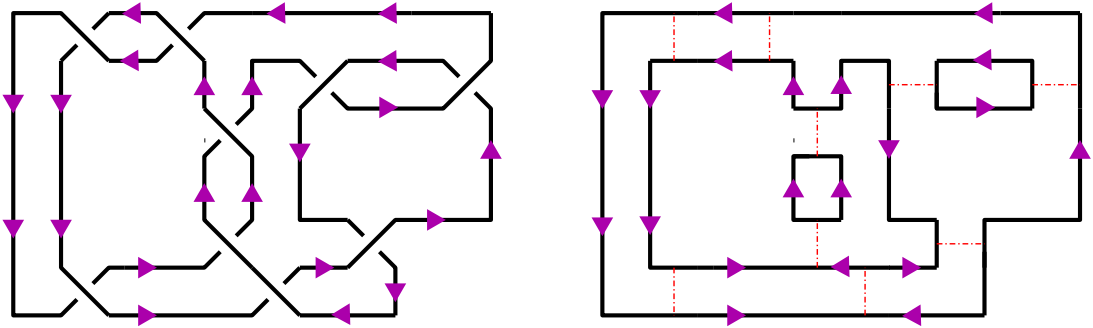}
\caption{Usual diagram for $9_{42}$ and the all-zero OT state, showing this diagram is + adequate.}
\label{figure-9_42-sigma0}
\end{figure}

For the third homology class, we would like to find a state cycle representative for the ``thick'' diagonal of $9_{42}$'s homology.  There is only one nontrivial homology entry for that diagonal, at bigrading (0, -1).  The s-invariant of $9_{42}$ is 0, which would give us a representative for this class if we had an explicit form of this generator.  But, we lack a general form for this generator in terms of the Khovanov chain generators.

However, the ``Seifert state'', gotten by smoothing the diagram according to the rules for constructing the canonical Seifert surface, is an adequate state.  And, marking it as a state cycle, with $v_-$ on the 0-tracing loops and $v_+$ on the 1-block, gives a state cycle $\alpha_s$ of the desired bigrading, (0, -1).  Since there is only one loop in its 1-block, and its quantum grading is $s-1$, Theorem \ref{thm-s-singleloop} guarantees this state cycle is nontrivial.  Figure \ref{figure-9_42-seifert} shows the associated OET state for this state cycle.  The ET state has marking $v_- \otimes v_- \otimes v_- \otimes v_+$, 4 loops, and 4 1-smoothings in a 9 crossing diagram with 4 negative and 5 positive crossings, for those wanting to verify the bigrading. 

\begin{figure}[ht!]
\centering
\includegraphics{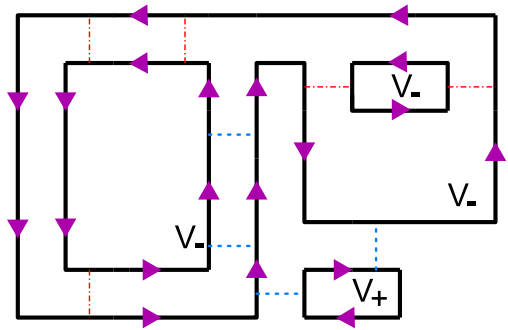}
\caption{OET for $\alpha_s$, the Seifert state cycle for $9_{42}$.}
\label{figure-9_42-seifert}
\end{figure}

To lift these three homology classes, we need to choose a set of braid-parallel arcs compatible with both state cycles, and a quasipositive braid.  The next two subsections will do this for two variants of this diagram, and analyze two families based on these choices of braid-parallel arcs.  

\subsection{Adding Positive Twists}
\label{example-9_42-postwist}

Looking over the OET states for $\alpha_0$ and $\alpha_s$, it is easy to find pairs of braid-parallel arcs compatible with both states:  Figure \ref{figure-9_42-bp-pair} shows the pair we will examine here.  Since there are only two arcs, whatever quasipositive braid we would use for the modification will in fact be Reidemeister equivalent to the braid word $\tau_1^n$:  $n$ positive half-twists between those two strands.  By the quasipositive modification theorems, this gives a simple family of thick links $K_n$, where $n$ is the number of half-twists. 

\begin{figure}[ht!]
\centering
\includegraphics{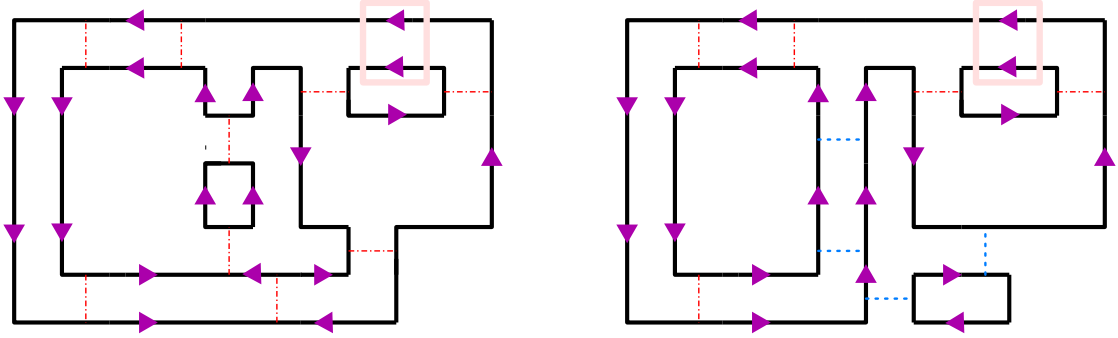}
\caption{OET states for $\alpha_0$ and $\alpha_s$, with a pair of compatible braid parallel arcs marked.}
\label{figure-9_42-bp-pair}
\end{figure}

Using SnapPea ~\cite{snappea}, one can check that the surgery link associated to this positive twisting is hyperbolic.  Using Thurston's hyperbolic Dehn surgery theorem ~\cite{thurston-gt3m}, it follows that all but a finite number of such $K_n$ must be hyperbolic.  So, all but a finite number of the $K_{2n}$ must be prime knots, showing that this family cannot be obtained by taking connect sums of something with a thick knot.

Furthermore, adding positive twists lends itself well to a recursive formula for the Jones polynomial of these knots.  Recall that the skein formula for the Jones polynomial $V(L)$ is given by:

\begin{equation}
q^{-1} V(L_+) - q V(L_-) = (q^{1/2} - q^{-1/2}) V(L_o) 
\label{eq-jones-skein}
\end{equation}

Here, $L_+$, $L_-$ and $L_o$ are the link where a crossing is circled and replaced by a positive crossing, a negative crossing, and the oriented smoothing, respectively.  To apply this to our $K_n$, note that changing one of the added positive half-twists to a negative half-twist will cancel the next positive half-twist, leaving $K_{n-2}$, while the oriented resolution of a positive half-twist just reduces the number of positive half-twists by 1, yielding $K_{n-1}$.  So, we can rewrite \eqref{eq-jones-skein} as the following recursion:

\begin{equation}
V(K_n) = q^2 V(K_{n-2}) + (q^{3/2} - q^{1/2}) V(K_{n-1}) 
\label{eq-postwist-recursion}
\end{equation}

Using KnotTheory' ~\cite{bn-knottheory}, one can verify the following calculations, which begin the recursion and suggest that the Jones polynomials of the $K_{2n}$ knots is alternating:

\begin{align*}
V(K_0) &= q^3 - q^2 + q - 1 + q^{-1} - q^{-2} + q^{-3} \\
V(K_1) &= q^{9/2} - q^{7/2} + q^{5/2} -2 q^{3/2} + q^{1/2} - 2 q^{-1/2} + q^{-3/2} - q^{-5/2} \\
V(K_2) &= q^6 - q^5 + q^4 - 2 q^3 + 2 q^2 - 2q + 2 - q^{-1} + q^{-2} \\
V(K_3) &= q^{15/2} - q^{13/2} + q^{11/2} - 2 q^{9/2} + 2 q^{7/2} - 3 q^{5/2} + 2 q^{3/2} - 2 q^{1/2} + q^{-1/2} - q^{-3/2} \\
V(K_4) &= q^9 - q^8 + q^7 - 2 q^6 + q^5 - 3 q^4 + 3 q^3 - 2 q^2 + q - 1 + q^{-1} 
\end{align*}

In fact, this pattern continues, leading to:

\begin{thm}
\label{thm-postwist-jones-alt}
The Jones polynomial of the knots $K_{2n}$ is alternating, in the following sense:  if $V(K_{2n}) = \sum a_t q^t$, then either all the $a_{2k}$ terms are positive and the $a_{2k+1}$ terms negative, or vice versa, with no zero coefficients within $ \Span(V(K_{2n}))$.
\end{thm}

\begin{proof}
The proof is by induction on the number of half-twists.  For convenience, we will consider everything in terms of the renormalization $P(K_n)(x) = V(K_n)(x^2)$: this way we can describe the odd $n$ case without worrying about all the half-powers.  In this renormalization, \eqref{eq-postwist-recursion} becomes:

\begin{equation}
P(K_n) = x^4 P(K_{n-2}) + (x^3 - x) P(K_{n-1}) 
\label{eq-p-recursion}
\end{equation}

Rewriting the base examples in terms of this renormalization, we have:

\begin{align*}
P(K_0) &= x^6 - x^4 + x^2 - 1 + x^{-2} - x^{-4} + x^{-6} \\
P(K_1) &= x^9 - x^7 + x^5 -2 x^3 + x - 2 x^{-1} + x^{-3} - x^{-5} \\
P(K_2) &= x^{12} - x^{10} + x^8 - 2 x^6 + 2 x^4 - 2x^2 + 2 - x^{-2} + x^{-4} \\
P(K_3) &= x^{15} - x^{13} + x^{11} - 2 x^{9} + 2 x^{7} - 3 x^{5} + 2 x^{3} - 2 x +  x^{-1} - x^{-3} \\
P(K_4) &= x^{18} - x^{16} + x^{14} - 2 x^{12} + x^{10} - 3 x^8 + 3 x^6 - 2 x^4 + x^2 - 1 + x^{-2} 
\end{align*}

Furthermore, as a consequence of the skein relations for the Jones polynomial, we know that $P(K_{2n})$ will have only even powers of $x$, while $P(K_{2n+1})$ will only have odd powers, since the respective link familes have an odd and even number of components, respectively.  To account for these two cases at once, we will consider the following induction hypothesis:

\begin{induction}
For all $n \leq k$, $P(K_n)$ has leading coefficient $+1$; if the leading power is $t$, then coefficients $a_j^n$ of $P(K_n)$ will be positive for $j \equiv t \pmod{4}$, and negative for $j \equiv t + 2 \pmod{4}$.  Coefficients of powers $i$ within $\Span(P(K_{k+1}))$ so that $i \equiv t \pmod{2}$ are nonzero.  The leading power $t$ of $P(K_n)$ will be $3 (n+2)$, and the trailing power will be $n-6$. 
\end{induction}

We will build this up by a few other inductive claims first.

\begin{claim}
$P(K_n)$ has leading power $3 (n+2)$ and leading coefficient $+1$.
\label{claim-leadingpower}
\end{claim}

By inspection, it is clear this holds for the base cases $P(K_0)$ and $P(K_1)$.  So, assume everything holds for $n \leq k$, and consider $P(K_{k+1})$.  The leading powers of $P(K_{k})$ and $P(K_{k-1})$ are $3k+6$ and $3k+3$ respectively, so by the skein formula, the highest powers each will contribute to $P(K_{k+1})$ are $3 + (3k+6) = 3k + 9$ and $4 + (3k + 3) = 3k + 7$ respectively.  Clearly $3k+9$ will be the highest power of the result, and because the sole such term comes from $P(K_{k})$ and is multiplied by $+1$, the resulting coefficient will remain $+1$ by induction.

\begin{claim}
$P(K_n)$ has trailing power $n-6$ and trailing coefficient $(-1)^n$.
\label{claim-trailingpower}
\end{claim}

By inspection, this holds for $P(K_0)$ and $P(K_1)$.  Now, assume this holds for all $ n \leq k $.  Observe that $-x P(K_k)$ will contribute lowest power $k-5$, while $x^4 P(K_{k-1})$ will contribute lowest power $k-3$.  So, the trailing term will come from $-x P(K_k)$, and the coefficient will be $-1$ multiplied by the trailing coefficient of $P(K_k)$, namely $(-1)*(-1)^k = (-1)^{k+1}$.

\begin{claim}
Denote the coefficient of $x^j$ in the expansion of $P(K_n)$ by $a_j^n$.  Then $ \| a_j^n \| \geq \| a_{j-3}^{n-1} \| $.
\label{claim-shiftbythree}
\end{claim}

\begin{proof}
Inspecting the first few examples, it is clear this holds for the base cases.  To account for the general case, we need to rewrite the recursion in terms of the coefficients.  For a fixed power $x^j$ and polynomial $P(K_n)$, \eqref{eq-p-recursion} becomes:

\begin{equation}
a_j^n = a_{j-4}^{n-2} + a_{j-3}^{n-1} - a_{j-1}^{n-1}
\label{eq-coefficient-recursion}
\end{equation} 

So, assume $ \| a_j^n \| \geq \| a_{j-3}^{n-1} \| $ holds for $n < k$.  Expanding by \eqref{eq-coefficient-recursion}, we get:

\begin{equation}
\| a_j^k \| = \| a_{j-4}^{k-2} + a_{j-3}^{k-1} - a_{j-1}^{k-1} \|
\label{eq-claim2-step1}
\end{equation}

We want to compare this to $ \| a_{j-3}^{k-1} \| $, which appears as a summand in \eqref{eq-claim2-step1}.  Let's first analyze the relative sign distribution inside the absolute value sign:  we will normalize so that $a_{j-3}^{k-1}$ is positive, multiplying the interior of the absolute value by $-1$ if needed.  $a_{j-1}^{k-1}$ is a coefficient from the same polynomial which differs by an $x^2$, so it must appear with the opposite sign from $a_{j-3}^{k-1}$.  So, $- a_{j-1}^{k-1}$ is positive when $a_{j-3}^{k-1}$ is positive.

On the other hand, $a_{j-4}^{k-2}$ comes from $P(K_{k-2})$, an $x^3$ factor lower term than $a_{j-1}^{k-1}$, which is negative.  We know from Claim \ref{claim-leadingpower} that the leading power of $P(K_{k-1})$ is $3k+3$, while that of $P(K_{k-2})$ is $3k$.  So, $a_{j-4}^{k-2}$ must match the negative sign of $a_{j-1}^{k-1}$.  This means that within the absolute value sign of \eqref{eq-claim2-step1}, when $a_{j-3}^{k-1}$ and  $- a_{j-1}^{k-1}$ are both positive terms, $a_{j-4}^{k-2}$ will be negative.

The upshot is that if $\| a_{j-1}^{k-1} \| \geq \| a_{j-4}^{k-2} \|$, then $ a_{j-4}^{k-2} - a_{j-1}^{k-1} \geq 0$.  In such a case, $ a_{j-4}^{k-2} - a_{j-1}^{k-1} $ and $ a_{j-3}^{k-1} $ have the same sign under the absolute value sign, giving us the following equality:

\begin{align}
\| a_{j-4}^{k-2} + a_{j-3}^{k-1} - a_{j-1}^{k-1} \| & = 
\| a_{j-4}^{k-2} - a_{j-1}^{k-1} \| + \| a_{j-3}^{k-1} \| \\
& \geq \| a_{j-3}^{k-1} \|
\end{align}

But, $\| a_{j-1}^{k-1} \| \geq \| a_{j-4}^{k-2} \|$ is simply an index shift of the induction hypothesis, completing the proof of this claim.
\end{proof}

We're now ready to tackle the last part of the induction.  The claims have already dealt with the leading and trailing coefficients and powers, so we have only to show the alternating signs of the coefficients.

Note first that all the terms from the $x^3 P(K_{k})$ expansion will match the parity of the terms from the $-x P(K_{k})$ expansion, since $x^3$ and $-x$ differ by an even power.  However, all terms from the expansion of $x^4 P(K_{k-1})$ will have the opposite sign parity, since the leading power is $3k+7$, which differs by 2 from the lead power of $x^3 P(K_{k})$.

To account for this discrepancy, we need to show that, for each power, we get a larger coefficient sum from $(x^3 - x) P(K_{k})$ than from $x^2 P(K_{k-1})$.  As in Claim \ref{claim-shiftbythree}, we will show this on a coefficient by coefficient level.  We want to show that the coefficient of $x^j$ in the expression $(x^3 - x) P(K_{k})$ is greater in absolute value than that of $x^k$ in the expansion of $x^2 P(K_{k-1})$.  At the coefficient level, this becomes:

\begin{equation}
\| a_{j-3}^k - a_{j-1}^k \| > \| a_{j-4}^{k-1} \|
\label{eq-alternating-coeffs}
\end{equation}

We then expand the lefthand side of this using the recursion relation, getting:    

\begin{align}
 & \| a_{j-3}^k - a_{j-1}^k \| \\
=& \| (a_{j-7}^{k-2} + a_{j-6}^{k-1} - a_{j-4}^{k-1}) - ( a_{j-5}^{k-2} + a_{j-4}^{k-1} - a_{j-2}^{k-1}) \| \\
=& \| a_{j-6}^{k-1} - 2 a_{j-4}^{k-1} + a_{j-2}^{k-1} + a_{j-7}^{k-2} - a_{j-5}^{k-2} \| \label{eq-abs-result}
\end{align}

Now, consider the relative parity of each term in \eqref{eq-abs-result}.  We would like to compare this to $a_{j-4}^{k-1}$, so let's assume that coefficient is negative (if it were positive, we could multiply the whole sum in the absolute value by $-1$).
 Clearly, $ - 2 a_{j-4}^{k-1}$ will be positive with this sign choice, and ultimately we would like to break off one of these copies as a separate absolute value part, as we did in Claim \ref{claim-shiftbythree}.
 
$a_{j-6}^{k-1}$ differs in power from this by 2, so it has the opposite sign, becoming positive.  The same thing happens to $ a_{j-2}^{k-1} $.  In contrast, $a_{j-7}^{k-2}$ comes from the next lower polynomial, and differs by a power of 3, so it has the same negative sign as $a_{j-4}^{k-1}$.  On the other hand, $a_{j-5}^{k-2}$ differs from $a_{j-7}^{k-2}$ by a power of 2, so must be positive:  this makes $-a_{j-5}^{k-2}$ a negative number.  In summary, our sign choice yields the following parity in the absolute value sign:

\begin{equation}
\| \underbrace{a_{j-6}^{k-1}}_+ - \underbrace{2 a_{j-4}^{k-1}}_- + \underbrace{a_{j-2}^{k-1}}_+ + \underbrace{a_{j-7}^{k-2}}_- - \underbrace{a_{j-5}^{k-2}}_+ \|
\label{eq-signanalysis}
\end{equation}

Termwise, the left three terms become positive, while the right two are negative.  If we are to use our earlier trick of separating out one of the ``positive'' terms, $- a_{j-4}^{k-1}$, we need to check that the remaining difference remains ``positive''.  Namely, we need to show:

\begin{equation}
\| a_{j-6}^{k-1} - a_{j-4}^{k-1} + a_{j-2}^{k-1} \| \geq
\| a_{j-7}^{k-2} - a_{j-5}^{k-2} \|
\label{eq-subclaim-arewedoneyet}
\end{equation}

But, by induction using \eqref{eq-alternating-coeffs}, we already know that:

\begin{equation}
\| a_{j-6}^{k-1} - a_{j-4}^{k-1} \| > \| a_{j-7}^{k-2} \|
\label{eq-savedbyinduction}
\end{equation}

And, Claim \ref{claim-shiftbythree} gives us the remaining piece, suitably shifted:

\begin{equation}
\| a_{j-2}^{k-1} \| \geq \| a_{j-5}^{k-2} \|
\label{eq-thisiswhywedidclaim2}
\end{equation}

Now, thanks to all the lefthand terms of \eqref{eq-subclaim-arewedoneyet} being positive, we can break apart the absolute value with equality, and then apply \eqref{eq-savedbyinduction} and \eqref{eq-thisiswhywedidclaim2} to verify the inequality of \eqref{eq-subclaim-arewedoneyet}:

\begin{align}
& \| a_{j-6}^{k-1} - a_{j-4}^{k-1} + a_{j-2}^{k-1} \| \\
= & \| a_{j-6}^{k-1} - a_{j-4}^{k-1}\|  +  \|a_{j-2}^{k-1} \| \\
> & \| a_{j-7}^{k-2} \| + \| a_{j-5}^{k-2}\| \\
= & \| a_{j-7}^{k-2} - a_{j-5}^{k-2} \|
\end{align}

The very last equality comes, again, from our earlier sign analysis in \eqref{eq-signanalysis}.  Now, since \eqref{eq-subclaim-arewedoneyet} holds, we can break up our earlier absolute value from \eqref{eq-abs-result} with equality, obtaining:

\begin{align}
 & \| a_{j-3}^k - a_{j-1}^k \| \\
=& \| a_{j-6}^{k-1} - 2 a_{j-4}^{k-1} + a_{j-2}^{k-1} + a_{j-7}^{k-2} - a_{j-5}^{k-2} \| \\
= & \| a_{j-6}^{k-1} - a_{j-4}^{k-1} + a_{j-2}^{k-1} + a_{j-7}^{k-2} - a_{j-5}^{k-2} \| + \| a_{j-4}^{k-1} \| \\
> & \| a_{j-4}^{k-1} \|
\end{align}

This verifies \eqref{eq-alternating-coeffs}, completing the claim that the coefficients of $P(K_{k+1})$ alternate in the appropriate fashion.  The remaining claim to verify is that the appropriate coefficients within $\Span(P(K_{k+1}))$ are nonzero.  This follows largely from Claim \ref{claim-shiftbythree}:  this tells us that coefficients of powers 3 higher than nonzero coefficients of $P(K_k)$ will be nonzero, since their magnitude is greater than that of nonzero coefficients of $P(K_k)$.  This covers coefficients of powers $(3k+6)+3) = 3k+9$ down to $(k-6) + 3 = k-3$, leaving only the trailing coefficient, of $x^{k-5}$ in question.  But by Claim \ref{claim-trailingpower}, we know this coefficient is nonzero.
\end{proof}

\begin{cor}
$K_{2n}$ cannot be detected as thick using the alternating Jones polynomial test of Khovanov.
\end{cor}

\begin{cor}
$\Span(V(K_{2n})) = 2n+6$
\end{cor}
\begin{proof}
In terms of $x$ coefficients, we know the leading power of $P(K_{2n})$ is $6n +6$ and the trailing coefficient is $2n-6$, so $\Span(P(K_{2n})) = 4n+12$.  It follows from the change of variables that $\Span(V(K_{2n})) = 2n+6$.
\end{proof}

\begin{thm}
For each $n$, $K_{2n}$ is not an adequate knot.
\end{thm}

\begin{proof}
Since the Jones polynomial of these knots begins and ends with $\pm 1$, we need to analyze the Kauffman polynomial to determine adequacy conditions.  What we will show is that $K_{2n}$ is not - adequate, and hence admits no adequate diagram.

Following the conventions of Thistlethwaite and Stoimenow, the Kauffman polynomial we consider will differ from that listed on Bar-Natan's Knot Atlas by the substitution $a \rightarrow a^{-1}$.  Recall that the Kauffman polynomial $F(K)(a, z)$ is defined by the following relations, where $D$ is the chosen diagram for $K$, $D_+$ is a diagram where a circled crossing is positive, $D_-$ is that same diagram with the crossing replaced by a negative crossing, $D_0$ replaces the crossing by a 0-smoothing, and $D_1$ replaces it by a 1-smoothing:

\begin{align}
F(K)(a, z) &=  a^{w(D)} \Lambda(D) \\
\Lambda(D_+) + \Lambda(D_-) &= z ( \Lambda(D_0) + \Lambda(D_1)) \label{lambda-recursion}\\
\Lambda(\text{positive RI twist}) &= a^{-1} \Lambda(|)\\
\Lambda(\text{negative RI twist}) &= a \Lambda(|) \\
\Lambda(\bigcirc) &= 1
\end{align}

Looking over the ingredients of \eqref{lambda-recursion}, we know that if $D_+$ focuses on one of the positive half-twists of $K_{n}$, then as with the Jones recursion, $D_-$ is $K_{n-2}$, and $D_0$ is $K_{n-1}$.  $D_1$ is the new quantity to understand, but it turns out that 1-smoothing the leftmost half-twist results in the mirror of $8_{19}$ with $n$ negative Reidemeister I twists added.  This leads to the following $\Lambda$ recursion:

\begin{equation}
\Lambda(K_n) = - \Lambda(K_{n-2}) + z \Lambda(K_{n-1}) + z a^{n} \Lambda(8_{19}!)
\label{recursion-k_n-lambda}
\end{equation}

The writhe of our diagram for $K_{n}$ is easily seen to be $1+n$, as $9_{42}$ has 4 negative crossings, and 5 positive crossings in its usual diagram.  We will be concerned with coefficients of $F(D)$, since this is a knot invariant and can give obstructions to the knot being - adequate.  So, translating \eqref{recursion-k_n-lambda} into the Kauffman polynomial, we obtain:

\begin{equation}
F(K_n) = - a^{1+n} \Lambda(K_{n-2}) + z a^{1+n} \Lambda(K_{n-1}) + z a^{1+2n} \Lambda(8_{19}!)
\label{recursion-k_n-kauffman}
\end{equation}

Let $[F(D)]_{(m, l)}$ denote the coefficient of $z^m a^l$ in $F(D)$.  By (2.10) of Stoimenow \cite{stoimenow-adequate},  if $l-m$ is the maximum integer so that $[F(D)]_{(m, l)}$ is nonzero, then if any such coefficient is negative, $K$ cannot be - adequate (note that Stoimenow's A-semiadequate corresponds to our convention for - adequate).  So, for the recursion, we just need to see how these maximal terms carry over for each term.

For our base case, KnotTheory' ~\cite{bn-knottheory} can be used to show that $F(K_0)$ is given by:

\begin{align*}
F(K_0) =& a^{-1} z^7 + z^7 a + a^-2 z^6 + z^6 a^2 + 2 z^6 - 5 a^{-1} z^5 - 5 z^5 a - 5 a^{-2} z^4 \\
& - 5 z^4 a^2 - 10 z^4 + 6 a^{-1} z^3 + 6 z^3 a + 6 a^{-2} z^2 + 6 z^2 a^2 + 12 z^2 - 2 a^{-1} z \\
&- 2 z a - 2 a^{-2} - 2 a^2 - 3
\end{align*}

So, the maximum $l-m$ term here is $-2 a^2$, and the knot is not - adequate.  The writhe for the standard diagram is $+1$, so dividing by $a$, one gets that the maximum $l-m$ term for $\Lambda(K_0)$ is $-2a$.  This will prove useful for the recursion.

Next, the Kauffman bracket for $K_1$ is:

\begin{align*}
F(K_1) =& 3 + 3 a^2 + a^4 - 2 a^{-1} z^{-1} - 3 a z^{-1} - a^3 z^{-1} + 8 a^{-1} z + 15 a z \\
&+ 8 a^3 z + a^5 z - 2 z^2 - 7 a^2 z^2 - 5 a^4 z^2 - 11 a^{-1} z^3 - 
 27 a z^3 - 16 a^3 z^3\\
&- 5 z^4 + 5 a^4 z^4 + 6 a^{-1} z^5 + 17 a z^5 + 
 11 a^3 z^5 + 5 z^6 + 4 a^2 z^6 - a^4 z^6\\
&- a^{-1} z^7 - 3 a z^7 - 2 a^3 z^7 - z^8 - a^2 z^8
\end{align*}

The maximum $l-m$ terms are $a^4 - a^3 z^{-1} + a^5 z$.  One of the terms is negative, so this link is also not - adequate.  The corresponding maximum $l-m$ terms from $\Lambda(K_1)$ are given by $a^2 - a z^-1 + a^3 z$. 

For the recursion, one can calculate that $\Lambda(8_{19}!)$ is:

\begin{align*}
\Lambda(8_{19}!) =&-5 - 5 a^{-2} - a^2 + 5 a^{-1} z + 5 a z + 10 z^2 + 10 a^{-2} z^2 - 
 5 a^{-1} z^3 \\
&- 5 a z^3 - 6 z^4 - 6 a^{-2} z^4 + a^{-1} z^5 + a z^5 + z^6 + a^{-2} z^6
\end{align*}

Thus, the term for which $l-m$ is maximal is $-a^2$, which has a negative coefficient.  We can now use these values to examine how the maximal $l-m$ terms behave under the recursion, for both $\Lambda$ and $F$:

\begin{claim}
For $n \geq 2$, the maximal $l-m$ term for $F(K_n)$ is $-z a^{2n+3}$, and the maximal $l-m$ term for $\Lambda(K_n)$ is $-z a^{n+2}$.
\end{claim}

\begin{proof}
Proof of the claim is a simple induction, using the recursion and the base values established above.  For our base case $n=2$, the recursion tells us that:

\begin{align*}
\max_{l-m} F(K_2) &\geq -a^3 \max_{l-m} \Lambda(K_0) + z a^3 \max_{l-m} \Lambda(K_1) + z a^5 \max_{l-m} \Lambda(8_{19}!) \\
&= -a^3(-2a) + z a^3 (za^3 - az^{-1} + a^2) + z a^5 (-a^2)\\
&= 2a^4 + z^2 a^6 - a^4 + z a^5 - z a^7 
\end{align*}

So, the maximum $l-m$ term for $F(K_2)$ is $-za^7$, and dividing by the writhe, the maximum $l-m$ term for $\Lambda(K_2)$ is $-za^4$, which matches the claim.  We need to analyze $K_3$ also before the induction:

\begin{align*}
\max_{l-m} F(K_3) &\geq -a^4 \max_{l-m} \Lambda(K_1) + z a^4 \max_{l-m} \Lambda(K_2) + za^7 \max_{l-m} \Lambda(8_{19}!) \\
&= -a^4 (a^2 - az^{-1} + a^3 z) + za^4 (-za^4 ) + za^7 (-a^2) \\
&= -a^6 + a^5 z^{-1} - a^7 z - z^2 a^8 - z a^9
\end{align*}

Again, the $-za^9$ term is maximal for $F(K_3)$, and correspondingly $-za^5$ is maximal for $\Lambda(K_3)$.  We can now catch the rest of the cases by induction:

\begin{align*}
\max_{l-m} F(K_n) &\geq -a^{1+n} \max_{l-m} \Lambda(K_{n-2}) + z a^{1+n} \max_{l-m} \Lambda(K_{n-1}) - z a^{2n+3} \\
&= -a^{1+n}(-za^{n}) + za^{1+n}(-za^{n+1}) - za^{2n+3} \\
&= za^{2n+1} -z^2 a^{2n+2} - z a^{2n+3}
\end{align*}

So, the maximal $l-m$ term for $F(K_n)$ is $-za^{2n+3}$; dividing by the writhe ($1+n$), we get that the maximal $l-m$ term for $\Lambda(K_n)$ is $-za^{n+2}$ as claimed.
\end{proof}

The claim tells us that the maximal $l-m$ term for the Kauffman polynomial of every $K_n$ has a negative coefficient, so the corresponding knot or link cannot admit a - adequate, and hence adequate, diagram.
\end{proof}

\subsection{Setup for modification by a 3-braid}

For a more complicated example of quasipositive modification, we need to find a trio of braid-parallel arcs.  While there is no such trio in the previous diagram, we can do a positive stabilization to get a new + adequate diagram for $9_{42}$ which now has a set of three braid-parallel arcs compatible with the two state cycles.  See Figure \ref{figure-9_42-bp-trio}.

\begin{figure}[ht!]
\centering
\includegraphics{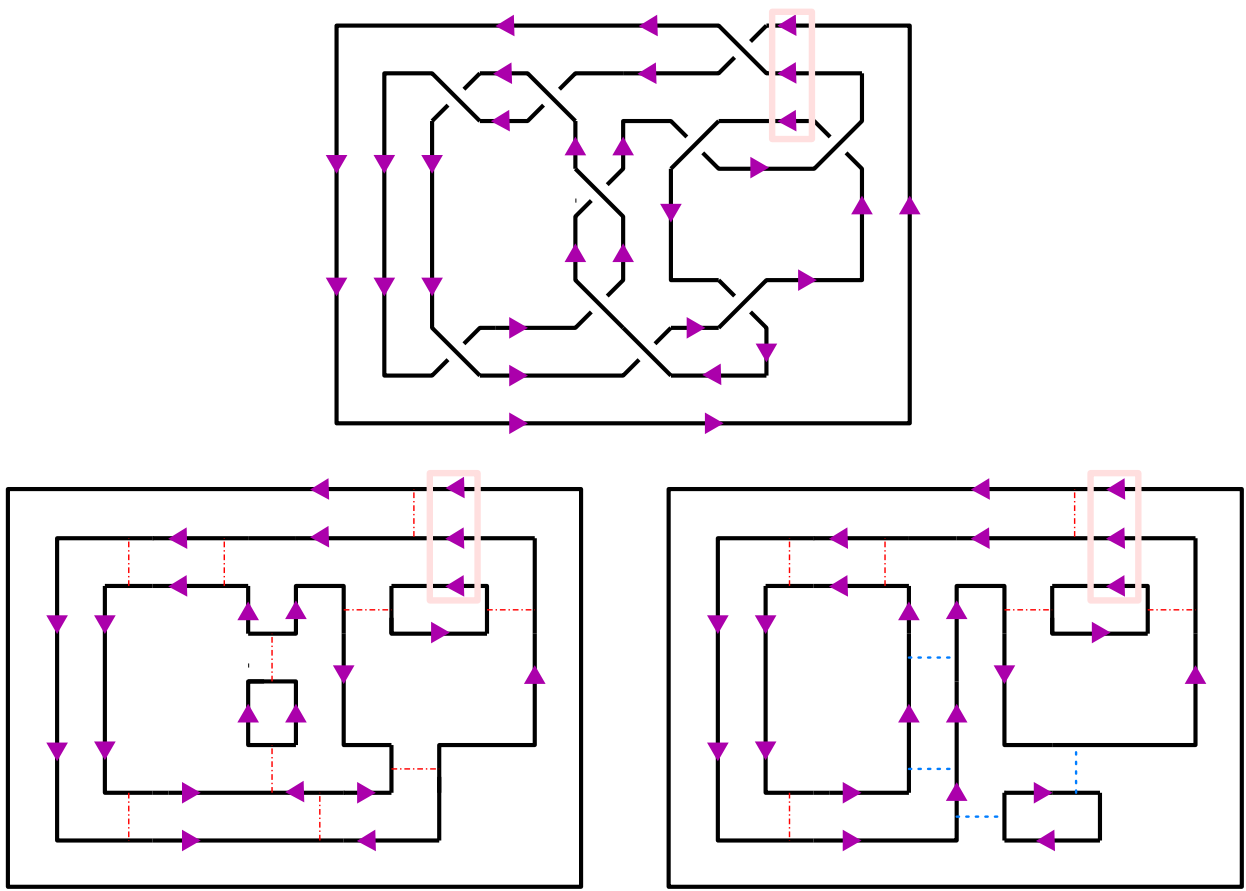}
\caption{Positive stabilization of $9_{42}$ and the new OET states for $\alpha_0$ and $\alpha_s$, with a trio of braid parallel arcs marked.}
\label{figure-9_42-bp-trio}
\end{figure}

Any quasipositive 3-braid can now be glued in for a quasipositive modification.  This gives a lot of variety in kinds of H-thick knots and links one can construct by modification of $9_{42}$; in the next two subsections, we will consider two such families.

\subsection{Conjugating a Positive Crossing}

For another family of thick knots which are prime, not adequate, and have alternating Jones polynomial, we can simply conjugate a positive crossing with a 3-braid word of choice.  For this example, we will conjugate by $ w_1 = \tau_1 \tau_2 \tau_1^{-1}$, to modify by the quasipositive braid word $ q_1 = \tau_1 \tau_2 \tau_1^{-1} \tau_2 \tau_1 \tau_{2}^{-1} \tau_1^{-1}$, choosing the same 3 braid-parallel arcs as in Figure \ref{figure-9_42-bp-trio}.  Note that inserting multiple copies of this braid word consecutively is isotopic by Reidemeister II cancellation to conjugating multiple positive crossings by $w_1$.

\begin{figure}[ht!]
\centering
\includegraphics{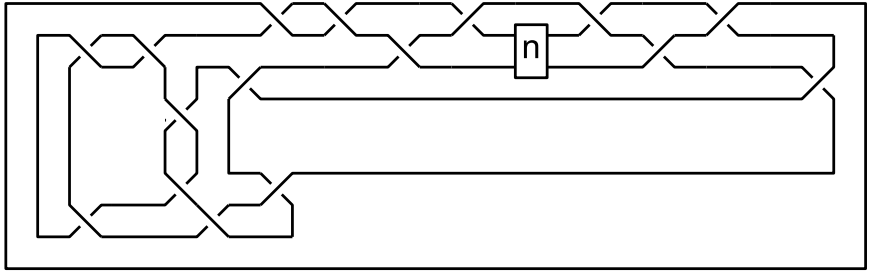}
\caption{Diagram of $K(9_{42}, q_1, n)$; the $n$ represents $n$ positive half-twists.}
\label{figure-9_42-word2}
\end{figure}

So, consider the family formed by successively doing quasipositive modification of $9_{42}$ by the 3-braid word $q_1$, denoted $K(9_{42}, q_1, n)$.  It turns out that this family is closely related to that of Subsection \ref{example-9_42-postwist}, though comparison of the Jones polynomial suggests these families are distinct.

Using SnapPea ~\cite{snappea}, one can again check that the surgery link associated to this positive twisting is hyperbolic.  Using Thurston's hyperbolic Dehn surgery theorem ~\cite{thurston-gt3m}, it follows that all but a finite number of such $K(9_{42}, q_1, n)$ must be hyperbolic.  Hence almost every member is prime.  Next, calculating the Jones polynomial via KnotTheory' ~\cite{bn-knottheory} for the first three members of the family, one sees that their Jones polynomials are alternating:

\begin{align*}
V(K(9_{42},q_1,1)) &= -2 q^{-5/2} + 3 q^{-3/2} - 5 q^{-1/2} + 5 q^{1/2} - 6 q^{3/2} + 
 6 q^{5/2} - 5 q^{7/2} + 4 q^{9/2} - 2 q^{11/2}\\
 &+ q^{13/2} - q^{15/2}\\
V(K(9_{42},q_1,2)) &= 2 q^{-2} - 4 q^{-1} + 7 - 9 q + 10 q^2 - 11 q^3 + 10 q^4 - 8 q^5 + 6 q^6 -  3 q^7 + 2 q^8 - q^9 \\
V(K(9_{42},q_1,3)) &= -2 q^{-3/2} + 4 q^{-1/2} - 8 q^{1/2} + 11 q^{3/2} - 14 q^{5/2} + 
 15 q^{7/2} - 15 q^{9/2} + 13 q^{11/2} - 10 q^{13/2} \\
 &+ 7 q^{15/2} - 4 q^{17/2} + 2 q^{19/2} - q^{21/2} \\
\end{align*}

This pattern holds in general, yielding:

\begin{thm}
The Jones polynomial of $K(9_{42}, q_1, n)$ is alternating, with no gaps.
\end{thm}

\begin{proof}
The proof is almost the same as that of Theorem \ref{thm-postwist-jones-alt}.  Doing the change of variables, the Jones polynomials of the first three members can be renormalized to:

\begin{align*}
P(K(9_{42}, q_1,1)) &= -2 x^{-5} + 3 x^{-3} - 5 x^{-1} + 5 x^{1} - 6 x^{3} + 
 6 x^{5} - 5 x^{7} + 4 x^{9} - 2 x^{11} 
 + x^{13} - x^{15}\\
P(K(9_{42}, q_1,2)) &= 2 x^{-4} - 4 x^{-2} + 7 - 9 x^2 + 10 x^4 - 11 x^6 + 10 x^8 - 8 x^{10} + 6 x^{12}
-  3 x^{14} + 2 x^{16} - x^{18} \\
P(K(9_{42}, q_1,3)) &= -2 x^{-3} + 4 x^{-1} - 8 x^{1} + 11 x^{3} - 14 x^{5} + 
 15 x^{7} - 15 x^{9} + 13 x^{11} - 10 x^{13} + 7 x^{15} - 4 x^{17}\\
 &+ 2 x^{19} - x^{21} \\
\end{align*}

The only difference of note is the recursion for the leading and trailing powers and coefficients:  in this family, the leading power is $12+3n$ and leading coefficient is $-1$; the trailing power is $n-6$ and trailing coefficient is $(-1)^n * 2$.  Otherwise, the proofs of the corresponding claims from Theorem \ref{thm-postwist-jones-alt} follow without change.
\end{proof}

\begin{cor}
$K(9_{42}, q_1, n)$ is not adequate.
\end{cor}
\begin{proof}
Because the trailing coefficient is $(-1)^n * 2$, $K(9_{42}, q_1, n)$ admits no adequate diagram by Proposition 1 of Lickorish and Thistlethwaite ~\cite{lickthistle-poly}.
\end{proof}

So, $ K(9_{42},q_1,2n)$ gives another infinite family of H-thick knots which is not detected by Khovanov's thickness criterion.  Presumably doing other such conjugations of a positive crossing will yield more families of this kind.

\subsection{Modification by the Mirror of $8_{20}$}

Now we will consider the case of gluing in multiple copies of the quasipositive mirror of $8_{20}$, which will form a different family of H-thick knots.  Call such a knot with $n$ copies of the mirror of $8_{20}$ glued in via this quasipositive modification $K(9_{42}, 8_{20}, n)$.  In actuality, this will not always return a knot:  in the case that $n \equiv 2 \pmod{3}$, the permutation of the strands gives a 3 component link, but otherwise $K(9_{42}, 8_{20}, n)$ will be a knot.

\begin{figure}[ht!]
\centering
\includegraphics{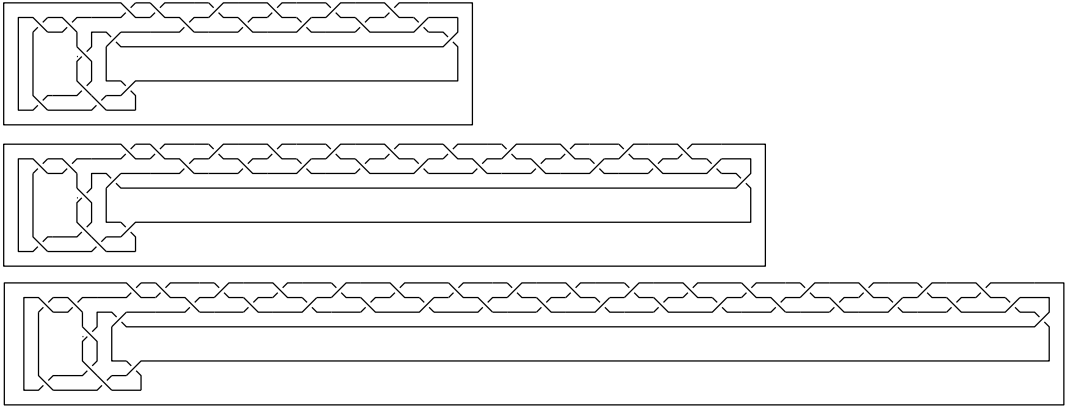}
\caption{First three members of family $K(9_{42}, 8_{20}, n)$.  Notice that the second member is a link, while the first and third are knots.}
\label{figure-9_42-qp8_20-trio}
\end{figure}

An interesting pattern in this family of H-thick links is that the width of successive members seems to be increasing.  Using JavaKh-v2, an update of Jeremy Green's program by Scott Morrison written for ~\cite{fgmw-spc4}, one can calculate that the width of the first four members is 3, 4, 5, and 6 respectively.  Furthermore, this increasing width is not coming from some well-hidden connect-sum operation hiding in the quasipositive modification:

\begin{prop}
For $n = \{ 1,2,3,4 \}$, $K(9_{42}, 8_{20}, n)$ is a prime link.  In particular, it is \emph{not} a nontrivial connect sum with a thick link as a summand.  
\end{prop}

\begin{proof}
The method of proof is to use prime tangle decomposition, following Lickorish ~\cite{lickorish-prime}.  We will first examine the case of $K(9_{42}, 8_{20}, 1)$; primeness for the other links will follow the same general pattern.

\begin{figure}[ht!]
\centering
\includegraphics{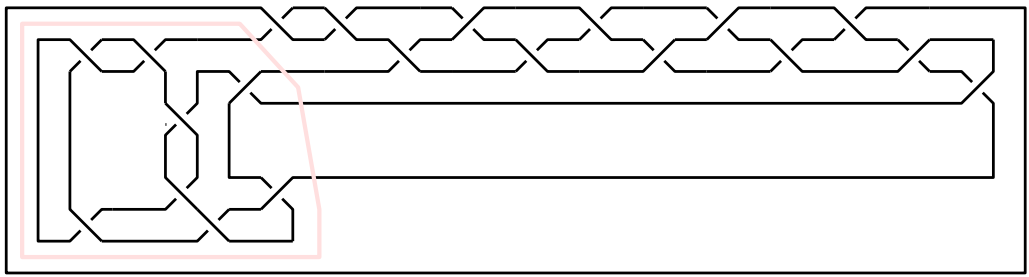}
\caption{$K(9_{42}, 8_{20}, 1)$ broken into prime tangles.  The enclosed tangle is Tangle A; the exterior is Tangle B.}
\label{figure-9_42-qp-prime-decomposition}
\end{figure}

Figure \ref{figure-9_42-qp-prime-decomposition} illustrates a breakdown of $K(9_{42}, 8_{20}, 1)$ into two tangles.  If we can show that both tangles are prime, then the original knot is prime by Theorem 1 of ~\cite{lickorish-prime}.

Tangle A is a variant on prime tangle (a) of Lickorish, which just has two extra positive half-twists added.  Lickorish's same argument shows Tangle A is prime - we add the untangle as shown in Figure \ref{figure-9_42-qp-tangle-a}, and get an unknot with no nontrivial summand.  As in (a), Tangle A itself is not the untangle, because one of its arcs is a knotted spanning arc of the ball; primeness of Tangle A follows.

\begin{figure}[ht!]
\centering
\includegraphics{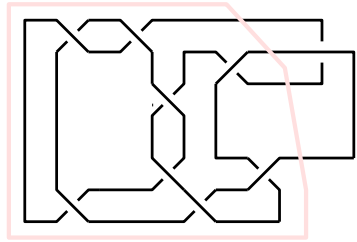}
\caption{If we add the untangle to Tangle A as shown here, we get an unknot with no nontrivial summand.  To see the unknotting, look at rightmost part and just start undoing twists via Reidemeister I.}
\label{figure-9_42-qp-tangle-a}
\end{figure}

\begin{figure}[ht!]
\centering
\includegraphics{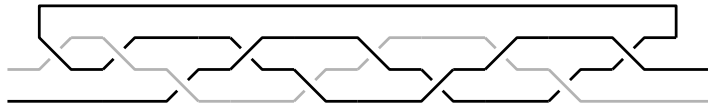}
\caption{An isotoped version of Tangle B.  The grey and black parts are both unknotted arcs, when viewed by themselves.}
\label{figure-9_42-qp-tangle-b-arcs}
\end{figure}

Primeness of Tangle B follows Lickorish's example (c).  Each of the two arcs of the tangle are unknotted, seen by examining Figure \ref{figure-9_42-qp-tangle-b-arcs}.  And, by adding an untangle to Tangle B as shown in Figure \ref{figure-9_42-qp-tangle-b-3bridge}, we get a 3-bridge knot, the mirror of $8_{21}$ (the Jones polynomial is enough to determine this identification).  Because the bridge number of this composition is higher than 2, Tangle B could not have been untangled.  So, Tangle B is prime.

\begin{figure}[ht!]
\centering
\includegraphics{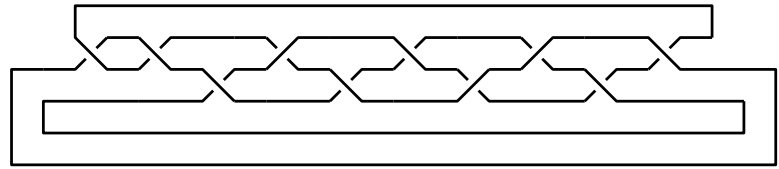}
\caption{Tangle B composed with an untangle.  This is a 10 crossing diagram, and only one knot of 10 crossings or less matches its Jones polynomial, so it is the mirror of $8_{21}$.}
\label{figure-9_42-qp-tangle-b-3bridge}
\end{figure}

Now let's consider similar tangle decompositions to prove the primeness in the cases $n=2,3,4$.  The idea is to make Tangle A the same, and consider what happens to the more complicated Tangle B cases.  The only new thing one must check is that adding an untangle to the new variants of Tangle B are not 2-bridge.  To show this, one can calculate the Khovanov homology for these links.  In each case, the Khovanov homology is H-thick, of width 3, 4, and 5 respectively.  But, 2-bridge links have alternating diagrams, which means their Khovanov homology must be H-thin.  So, the bridge index of these two links is higher than 2, and the new Tangle B's are prime, as desired.  
\end{proof}

\begin{remark}
Note that via this prime tangle decomposition, we end up using H-thickness of some subtangles to prove these knots and links are prime.  A better understanding of why the width seems to increase when adding the mirror of $8_{20}$ would also give a proof that this full family of links is prime.
\end{remark}

On the other hand, these links have nonalternating Jones polynomial, so their thickness (but not increasing width) can be detected by Khovanov's thickness criteria.  For reference, the Jones polynomials of the first four in this family are listed below, calculated using KnotTheory' ~\cite{bn-knottheory}:

\begin{align*}
V(K(9_{42}, 8_{20}, 1)) &= 2 q^{-2} - 3q^{-1} + 5 - 5 q + 4 q^2 - 4 q^3 + 2 q^4 - q^5 + q^7 \\
V(K(9_{42}, 8_{20}, 2)) &= - q^{-3} + 3q^{-2} - 5q^{-1} + 9 - 11 q + 16 q^2 - 17 q^3 + 19 q^4 - 17 q^5\\
 &+ 14 q^6 - 10 q^7 + 6 q^8 - 2 q^9 + q^{11} - q^{12}\\
V(K(9_{42}, 8_{20}, 3)) &= q^{-4} - 4q^{-3} + 10q^{-2} - 20q^{-1} + 34 - 51 q + 69 q^2 - 85 q^3 + 98 q^4\\
&- 104 q^5 + 105 q^6 - 97 q^7 + 84 q^8 - 67 q^9 + 48 q^{10} - 31 q^{11} + 16 q^{12} - 6 q^{13}\\
&+ 2 q^{15} - 2 q^{16} + q^{17} \\
V(K(9_{42}, 8_{20}, 4)) &= - q^{-5} + 5q^{-4} - 15q^{-3} + 35q^{-2} - 68q^{-1} + 118 - 183 q + 264 q^2 - 353 q^3 \\
&+ 444 q^4 - 526 q^5 + 587 q^6 - 620 q^7 + 619 q^8 - 585 q^9 + 522 q^{10} - 438 q^{11} + 343 q^{12} \\
&- 250 q^{13} + 166 q^{14} - 98 q^{15} + 49 q^{16} - 17 q^{17} + q^{18} + 5 q^{19} - 5 q^{20} + 3 q^{21} - q^{22} \\
\end{align*}

For completeness, one might ask whether these links have adequate diagrams.  Unfortunately, the Kauffman polynomial, which provides most of the obstructions to being adequate, is too computationally intensive to calculate even for the 20 crossing diagram of the $n=1$ case.  For $n=1$, the last coefficient of the Jones polynomial is not $\pm 1$, guaranteeing there is no adequate diagram for this first knot by Proposition 1 of Lickorish and Thistlethwaite ~\cite{lickthistle-poly}.  But, subsequent Jones polynomials for the next 3 cases all have first and last Jones coefficient $\pm 1$, leaving only the Kauffman polynomial obstructions.

\begin{conj}
For each natural number $n$, $K(9_{42}, 8_{20}, n)$ admits no adequate diagram.
\end{conj}

\subsection{Other Base Knots}

$9_{42}$ is not the only valid base knot for constructing H-thick knots by quasipositive modification:  $10_{132}$, $10_{136}$, and $10_{145}$ all have the same setup of a + adequate diagram, and a ``thick'' state cycle representative in the third diagonal, for a knot of width 3.  Note that for all these knots, $9_{42}$ included, the ``thick'' state cycle representative has bigrading $(0, s-1)$, where $s$ is Rasmussen's $s$ invariant.  Also, in each case, the $s$ invariant is smaller than the signature of the knot.

One can also find examples with a slightly different setup, by direct computation.  $10_{161}$ has a - adequate diagram for its standard minimal diagram, but there is a state cycle representative for its lowermost corner homology class at $(-9, -23)$.  Its width is 3, so even though the diagram is not + adequate, the same methods will give us 2 diagonals from the Lee pairing of this data.  And, there is a third off diagonal state cycle which is nontrivial, at bigrading $(-3,-13)$.  For these state cycles, nontriviality was checked by a Java program written by the author.

In general, though, this program is limited to calculating nontriviality for diagrams of 10 crossings or less, restricting the number of H-thick base knots that can be examined by direct computation.  Better methods of checking nontriviality for state cycles are needed, if one is to extend the selection of base knots for such thick families.

\bibliographystyle{hep}
\bibliography{knotbib}

\begin{thebibliography}{FGMW}

\bibitem[Baa05]{baader-thesis}
S.~Baader,
\newblock \textsl{ Notes on Quasipositivity and Combinatorial Knot Invariants},
\newblock PhD thesis, Universit{\aa}t Basel, 2005.

\bibitem[BN02]{bn-khov-1}
D.~Bar-Natan, \textsl{ On {K}hovanov's categorification of the {J}ones
  polynomial},
\newblock Algebr. Geom. Topol. \textbf{ 2}, 337--370 (electronic) (2002),
  {math.QA/0201043}.

\bibitem[BNM]{bn-atlas}
D.~Bar-Natan and S.~Morrison,
\newblock Knot {A}tlas,
\newblock \url{http://katlas.math.toronto.edu}.

\bibitem[BNMG]{bn-knottheory}
D.~Bar-Natan, S.~Morrison and J.~Green,
\newblock Knot{T}heory',
\newblock \url{http://katlas.math.toronto.edu/wiki/Setup}.

\bibitem[BP08]{bo-openbook}
J.~A. Baldwin and O.~Plamenevskaya,
\newblock Khovanov homology, open books, and tight contact structures, 2008.

\bibitem[FGMW]{fgmw-spc4}
M.~Freedman, R.~Gompf, S.~Morrison and K.~Walker,
\newblock Man and machine thinking about SPC4.

\bibitem[Jac04]{jacobsson-cob}
M.~Jacobsson, \textsl{ An invariant of link cobordisms from {K}hovanov
  homology},
\newblock Algebr. Geom. Topol. \textbf{ 4}, 1211--1251 (electronic) (2004).

\bibitem[Kho03]{khov-patterns}
M.~Khovanov, \textsl{ Patterns in knot cohomology. {I}},
\newblock Experiment. Math. \textbf{ 12}(3), 365--374 (2003),
  {math.QA/0201306}.

\bibitem[Lee05]{lee}
E.~S. Lee, \textsl{ An endomorphism of the {K}hovanov invariant},
\newblock Adv. Math. \textbf{ 197}(2), 554--586 (2005), {math.GT/0210213}.

\bibitem[Lic81]{lickorish-prime}
W.~B.~R. Lickorish, \textsl{ Prime knots and tangles},
\newblock Trans. Amer. Math. Soc. \textbf{ 267}(1), 321--332 (1981).

\bibitem[Low09]{lowrance-thick-rational}
A.~Lowrance,
\newblock The Khovanov width of twisted links and closed 3-braids, 2009.

\bibitem[LT88]{lickthistle-poly}
W.~B.~R. Lickorish and M.~B. Thistlethwaite, \textsl{ Some links with
  nontrivial polynomials and their crossing-numbers},
\newblock Comment. Math. Helv. \textbf{ 63}(4), 527--539 (1988).

\bibitem[MO07]{manolescu-2007}
C.~Manolescu and P.~Ozsvath,
\newblock On the Khovanov and knot Floer homologies of quasi-alternating links,
  2007.

\bibitem[Oza06]{ozawa-state}
M.~Ozawa,
\newblock Essential state surfaces for knots and links, 2006.

\bibitem[Pla06]{olga}
O.~Plamenevskaya, \textsl{ Transverse knots and {K}hovanov homology},
\newblock Math. Res. Lett. \textbf{ 13}(4), 571--586 (2006), {math.GT/0412184}.

\bibitem[Ras04]{rasmussen-slice}
J.~A. Rasmussen,
\newblock Khovanov homology and the slice genus, 2004.

\bibitem[Sto06]{stosic-2006}
M.~Stosic,
\newblock Homology of torus links, 2006.

\bibitem[Sto07]{stoimenow-adequate}
A.~Stoimenow,
\newblock On the crossing number of semiadequate links,
\newblock \url{http://mathsci.kaist.ac.kr/~stoimeno/papers.html}, 2007.

\bibitem[Suz06]{suzuki-2006}
R.~Suzuki,
\newblock Khovanov homology and Rasmussen's s-invariants for pretzel knots,
  2006.

\bibitem[Thu02]{thurston-gt3m}
W.~P. Thurston,
\newblock The {G}eometry and {T}opology of {T}hree-{M}anifolds,
\newblock \url{http://www.msri.org/publications/books/gt3m/}, 2002.

\bibitem[Tur06]{turner-2006}
P.~Turner,
\newblock A spectral sequence for Khovanov homology with an application to
  (3,q)-torus links, 2006.

\bibitem[Wee]{snappea}
J.~Weeks,
\newblock Snap{P}ea,
\newblock \url{http://www.geometrygames.org/SnapPea/index.html}.

\end{thebibliography}
\end{document}